\documentclass[11pt,reqno]{amsart}
\usepackage[left=1.3in, right=1.2in, top=1in, bottom=1in, includefoot]{geometry}
\usepackage{amssymb,amsmath}
\usepackage{bm}
\usepackage{mathrsfs}
\usepackage{colonequals}
\usepackage[english]{babel}
\usepackage[all]{xy}
\usepackage{tikz}

\numberwithin{equation}{section}

\newtheorem{theorem}{Theorem}[section]

\newtheorem{lemma}[theorem]{Lemma}
\newtheorem{proposition}[theorem]{Proposition}
\newtheorem{corollary}[theorem]{Corollary}
\theoremstyle{remark}

\theoremstyle{definition}

\newtheorem{remark}[theorem]{Remark}

\def\R{{\mathbb{R}}}
\def\T{{\mathbb{T}}}

\def\Z{{\mathbb{Z}}}

\def\supp{{\text{supp}}}

\newcommand{\jb}[1]{\langle #1 \rangle}
\renewcommand{\hat}[1]{\widehat{#1}}

\newcommand{\ind}{\mathbf{1}} 

\DeclareMathOperator{\sgn}{sgn}

\renewcommand{\H}{\mathcal{H}}

\newcommand{\F}{\mathcal{F}}

\newcommand{\wt}{\widetilde}
\newcommand{\wh}{\widehat}

\newcommand{\dx}{\partial_x}

\newcommand{\dt}{\partial_t}

\newcommand{\Plo}{P_{\textup{lo}}}
\newcommand{\PLO}{P_{\textup{LO}}}
\renewcommand{\Phi}{P_{\textup{hi}}}
\newcommand{\PHI}{P_{\textup{HI}}}
\newcommand{\Pphi}{P_{+\textup{hi}}}
\newcommand{\Pmhi}{P_{-\textup{hi}}}

\title[Unconditional uniqueness for the Benjamin-Ono equation] 
{Unconditional uniqueness for the Benjamin-Ono equation}

\author[R. Mosincat]{Razvan Mosincat}
\address{
Department of Mathematics\\ University of Bergen\\ Postbox 7800\\ 5020 Bergen\\ Norway}
\email{Razvan.Mosincat@UiB.no}

\author[D. Pilod]{Didier Pilod}
\address{
Department of Mathematics\\ University of Bergen\\ Postbox 7800\\ 5020 Bergen\\ Norway}
\email{Didier.Pilod@UiB.no}

\date{\today}


\begin{document}

\selectlanguage{english}

\maketitle

\begin{abstract}
We study the unconditional uniqueness  
of solutions to the Benjamin-Ono equation with initial data in $H^{s}$, both on the real line and on the torus. 
We use the gauge transformation of Tao 
and two iterations of normal form reductions 
via integration by parts in time. 
By employing a refined Strichartz estimate 
we establish the result below the regularity threshold $s=1/6$. 
As a by-product of our proof, we also obtain a nonlinear smoothing property on the gauge variable at the same level of regularity.
\end{abstract}


\section{Introduction}

We consider the Benjamin-Ono equation (BO)
\begin{equation}
\label{BO}
 \dt u + \H\dx^2  u = \dx(u^2) \ ,
\end{equation}
where $u=u(t,x)$ is a real-valued function, $t \in \mathbb R$, $x \in \mathbb R$ or $\mathbb T$ 
and $\H$ is the Hilbert transform, together with 
the initial condition 
\begin{equation}
\label{inicond}
u|_{t=0}=u_0\,,
\end{equation}
where the initial data  $u_0$ lies in the Sobolev space $H^s(\R):=H^s(\R;\R)$ or \-$H^s(\T):=H^s(\mathbb T;\mathbb R)$\footnote{We will also use $H^s$ to denote both $H^s(\R)$ and $H^s(\T)$ when the statements apply in both cases.}. 
This equation appears as a model for the propagation of unidirectional internal waves in stratified fluids 
\cite{Benjamin,Ono}
and it is completely integrable \cite{AblowitzFokas}. We refer the reader to  \cite{SautReview} for a review of the derivation of this model as well as an up-to-date survey of the literature on BO and related equations. 

The well-posedness of BO provides analytical challenges at various regularity levels $s$, 
mainly due to the presence of the spatial derivative in the nonlinearity 
and weak dispersive properties in the linear part 
-- see
\cite{SautBO79, Iorio86, ABFS,PonceBO, KochTzvetkovIMRN03, KochTzvetkovIMRN05, KenigKoenig, TaoBO04, BurqPlanchonBO, IonescuKenigBO, MolinetPilodAPDE12, IfrimTataruBO,MPVaIHP18} 
in the real line case and \cite{MolinetPeriodicBOenergysp,MolinetAJM08,MolinetPilodAPDE12, GTcpam2021, GKTflow, GKTwp} in the periodic case. 
Nowadays, it is known that BO is (globally in time) well-posed in $H^s$, for any $s\ge 0$.  
This result was first established by Ionescu and Kenig \cite{IonescuKenigBO} in the Euclidean case and by Molinet \cite{MolinetAJM08} in the periodic case. We also refer to the papers of 
Molinet and Pilod \cite{MolinetPilodAPDE12} and of Ifrim and Tataru \cite{IfrimTataruBO}\footnote{The method in \cite{IfrimTataruBO} also provides long time asymptotics for solutions emanating from small initial data.}
for other proofs.  
The solution constructed by \cite{IonescuKenigBO,MolinetAJM08,MolinetPilodAPDE12,IfrimTataruBO} is guaranteed to be unique either in the class of limits of classical solutions 
or under some additional condition on (some transformation of) the solution itself. 
Therefore the uniqueness of solution remains conditional, dependent on the method used. 

Below $L^2(\T)$,  
by using the Lax pair formulation of \eqref{BO}, 
G\'{e}rard, Kappeler, and  Topalov \cite{GTcpam2021,GKTflow, GKTwp} showed that BO  
in the periodic setting is (globally in time) well-posed in the sense that the solution map 
(defined for smooth data) continuously extends to $H^s(\T)$  for $-\frac12<s<0$ and that no such extension exists for $s\le -\frac12$, even if the mean-value of the solution is prescribed. 
We refer the reader to their survey paper~\cite{GKTreview} for a precise statement of these results 
as well as other powerful applications of the nonlinear Fourier transform such as the construction of periodic and quasiperiodic solutions to \eqref{BO}-\eqref{inicond}.  
Indeed, $s_{\text{crit}}=-\frac12$ is a natural threshold for the well-posedness of BO 
as indicated by the invariance of the homogeneous Sobolev norm with respect to the scaling symmetry of the equation. 

To this date the well-posedness of BO on the real line below $L^2(\mathbb R)$ remains an open problem. Note however that the direct scattering problem was solved in \cite{WuSIMA17} and that the complete integrability of BO restricted to $N$-soliton manifolds has been recently proved in \cite{SunCMP21}.
We also mention here that the techniques developed in \cite{KillipVisanZhang} were applied for BO  in  \cite{TalbutBO} showing that there exist conservation laws of  Sobolev norms at negative regularity (namely $-\frac12<s<0$) for classical solutions to \eqref{BO}-\eqref{inicond}. 
Further in this direction, 
the method of perturbation determinants was successfully employed \cite{KillipVisanKdV} to show that the Korteweg-de Vries equation (KdV) is well-posed in $H^{-1}(\mathbb R)$.

BO has a quasilinear nature in that the dependence of solution  on the initial data is merely continuous in the $H^s$-topology, 
even at high regularity. Indeed, this was first pointed out by Molinet, Saut, and Tzvetkov \cite{MolinetSautTzvetkov} 
showing that  the $C^2$ continuity of the solution map fails for any $s\in\R$. 
Furthermore,  even 
uniform continuity (on bounded subsets of $H^s$) fails for any $s>0$ and $s<-\frac12$ in the Euclidean case 
due to  \cite{KochTzvetkovIMRN05,BiagioniLinares} 
and  for any $s>-\frac12$ in the periodic case due to  \cite{GKTbirkhoff}. 
This property of the Benjamin-Ono equation tells us that the nonlinearity is in fact non-perturbative since it  prohibits a direct application of fixed point arguments. 
To improve the nonlinearity, Tao \cite{TaoBO04} considered a variant of the Cole-Hopf transformation, i.e.  
\begin{equation}
\label{defn:gaugeTao}
w := \dx \Pphi(e^{-iF}) \,,
\end{equation}
where $F$ is a spatial anti-derivative of $u$ and $\Pphi$ denotes 
the Littlewood-Paley projection to positive high frequencies. 
Consequently, one works with an equation for $w$ which is no longer in closed form (see \eqref{eqw:intro} below), 
but has the advantage of having a milder nonlinearity. 
This idea and various refinements turned out to be central in the papers 
\cite{BurqPlanchonBO, IonescuKenigBO, MolinetAJM08, MolinetPilodAPDE12,IfrimTataruBO} and it is also key to our work.

The question we address in this paper is that of \emph{unconditional uniqueness} of solutions to BO, 
i.e. whether for given initial data $u_0\in H^s$ the  solution $u$ to \eqref{BO} is 
\emph{unique in the entire space $C(\R; H^s)$}. 
In the affirmative, 
the uniqueness statement in the well-posedness theory can be upgraded, 
namely it now holds without restricting the solution to a resolution subspace specific 
to some particular  method(s). To be precise, 
by \emph{solution} to the initial-value problem \eqref{BO}-\eqref{inicond} 
we mean a continuous function in time with values in $H^s$ 
satisfying the integral (Duhamel) formulation
\begin{equation}
\label{BO:duhamel}
u(t) = e^{t\H\dx^2} u_0 + \int_0^t e^{(t-t')\H\dx^2} \dx (u(t')^2) dt'
\end{equation}
in the sense of (tempered) distributions, for all times $t$. 

For nonlinear dispersive PDEs, 
the study of unconditional well-posedness goes back to the work of Kato \cite{KatoNLSuniqueness}  
who was the first to address the question for the nonlinear Schr\"{o}dinger equation (NLS). 
Since then the unconditional well-posedness for NLS was further improved, see \cite{FPT,GuoKwonOh,HerrSohinger,KishimotoNLS,KwonOhYoon}
and studied for various other  nonlinear dispersive PDEs, 
see e.g. \cite{BabinIlyinTiti,ZhouKdV} for KdV, 
\cite{KwonOhIMRN,MolinetPilodVentoIbero18,MPVjapan,KwonOhYoon} for the modified KdV equation, \cite{ErdoganTzirakisDNLShalfline,MosincatYoon} for the derivative NLS equation, and 
\cite{KishimotoMBO} for the periodic modified Benjamin-Ono equation.

The uniqueness of solution problem for the Benjamin-Ono equation received attention  in several papers. 
We mention here that, in the Euclidean setting, the $L^2$-well-posedness result  in \cite{IonescuKenigBO} ensured uniqueness only in the class of limits of smooth solutions, while the approach of \cite{BurqPlanchonBO} rendered  unconditional uniqueness for data  in $H^{\frac12}(\mathbb R)$ (see \cite{BurqPlanchonBOenergy06}). 
This result was further improved in \cite{MolinetPilodAPDE12} to unconditional uniqueness in $H^s(\mathbb R)$ for any $s>\frac14$ and a conditional uniqueness statement for any $s>0$. 
The method in \cite{MolinetPilodAPDE12} also yielded unconditional   uniqueness in $H^s(\mathbb T)$, $s \ge \frac12$. 
More recently, Kishimoto showed in \cite{KishimotoBO} that  BO is unconditionally well-posed in $H^s(\mathbb T)$ for any $s>\frac16$. 

At an expeditious investigation, 
the regularity $s=\frac16$ appears to be a  possible threshold 
for the unconditional well-posedness of BO.  
Indeed, after renormalizing the equation for $w$, 
one encounters a variant of the NLS equation (a cubic term plus some other nonlinearities - see the equation \eqref{BOto cubicNLS})\footnote{Such a cubic NLS-type structure also appears in \cite{IfrimTataruBO}, where the authors performed two normal form transformations, the first one in the spirit of Shatah \cite{Sha} 
and the second one in the spirit of the gauge transformation of Tao \cite{TaoBO04}.} .
Therefore, for the renormalized equation, 
the largest possible space $C(\R; H^s)$ for the solution $u$ (and thus for $w$) 
in which one can make sense of the nonlinearity  as a spatial distribution is given by $s=\frac16$, 
courtesy of the Sobolev embedding $H^{\frac16} \subset L^3$. 
Note, however, that for the original equation \eqref{BO}
one can easily make sense of the nonlinearity as soon as $s\ge 0$.  
Thus it was unclear whether the cubic nonlinearity determines a regularity restriction for the unconditional well-posedness of BO.

In this article, we answer this question by showing that the regularity for the  unconditional well-posedness  of BO in $H^s$ 
can be further pushed down past $s=\frac16$. 
We state the main result of this paper which holds both on the line and on the torus. 

\begin{theorem}
\label{mainthm}
Let $3-\sqrt{33/4} < s\le\frac14$ and $u_0\in H^s=H^s(\mathbb R)$ or $H^s(\mathbb T)$.  
Then, there exists a unique solution $u\in C(\R; H^s)$  to the Benjamin-Ono equation \eqref{BO} with \eqref{inicond}.  
\end{theorem}

Note that $3-\sqrt{33/4} \sim 0.128 <\frac16$. We believe that this lower bound on $s$ is simply a technical restriction that appears in our main nonlinear estimates which hold under a quadratic restriction on $s$ (see {Corollary~\ref{cor:3p12} below).
In fact, we speculate that BO is unconditionally well-posed down to $L^2$, possibly missing the end-point $s=0$. 

As a by-product of our proof, we also obtain a nonlinear smoothing property for the gauge variable $w$, both on the line and on the torus, which may be of independent interest. 

%

\begin{corollary} \label{maincoro}
Let $3-\sqrt{33/4} < s \le \frac14$.
 \textup{(i)} If $u_0 \in H^s=H^s(\mathbb R)$, then 
 there exists $\delta>0$ such that for all $T>0$, 
\begin{equation}
\big\|w(t)- e^{it\dx^2}w_0\big\|_{L_T^{\infty}H^{s+\delta}}  \le C(T, \|u_0\|_{H^s}) <\infty \,,
\end{equation}
where $w$ is the gauge variable defined in \eqref{defn:gaugeTao} 
corresponding to the solution $u\in C(\R; H^s)$  to the Benjamin-Ono equation \eqref{BO} emanating from $u_0$.\\ 
\textup{(ii)} If 
 $u_0 \in H^s=H^s(\mathbb T)$, assume that $\int_{\T} u_0(x)dx =0$ 
and let $m_0:=\frac{1}{2\pi} \int_{\T}u_0^2(x)dx$. 
Then,  there exists $\delta>0$ such that for all $T>0$, 
\begin{equation}
\big\|w(t)- e^{it(\dx^2+m_0)}w_0\big\|_{L_T^{\infty}H^{s+\delta}}  \le C(T, \|u_0\|_{H^s}) <\infty \,,
\end{equation}
where $w$ is the gauge variable defined in \eqref{gaugetrT} 
corresponding to the solution $u\in C(\R; H^s)$  to the Benjamin-Ono equation \eqref{BO} emanating from $u_0$. 
\end{corollary}

\begin{remark}
In the periodic case, a similar nonlinear smoothing was shown by Isom, Mantzavinos, Oh and Stefanov in \cite{IMOS}  for $s>\frac16$, by using the Fourier restriction norm method. More recently, this result has been extended up to $L^2(\T)$ by G\'erard, Kappeler, and Topalov in \cite{GKTsmoothing} by using the complete integrability structure of BO. Note that  the gain of regularity in \cite{GKTsmoothing} is $\delta=2s$ and is also proved to be sharp. 
\end{remark}

We now briefly describe our approach to proving the above unconditional well-posedness result for BO. 
We first renormalize the equation \eqref{BO} by employing the gauge transformation \eqref{defn:gaugeTao} of Tao \cite{TaoBO04} in order to remove the worst high-low frequency interaction in the nonlinearity. 
At this point $w$ satisfies a Schr\"{o}dinger equation with two quadratic nonlinearities, 
one of them being negligible (as shown by Lemma~\ref{est:negl1:v1}): 
\begin{equation}
\label{eqw:intro}
\dt w - i\dx^2 w 
	= -2 \Pphi\dx\big[\dx^{-1}w \cdot P_-\dx u\big]  + \text{``negligible term''} \,.
\end{equation}
Also, by following the idea used in \cite[Section~4]{KishimotoBO}, in Lemma~\ref{lem:diffinus} 
we establish $H^s$-estimates for the difference of two solutions to BO in terms of 
the difference of the corresponding gauge transformations, for any $s\ge 0$. 
It then remains to establish  reverse estimates with constants that can be taken arbitrarily small. 
To this purpose, 
the idea is to further  renormalize the main nonlinearity in \eqref{eqw:intro}
via the fairly elementary method of integration by parts in the temporal variable. 
By considering the Van der Pole change of variables on the Fourier side, i.e. 
\begin{equation}
\label{VdP}
\widetilde{u}(t,\xi) := \mathcal{F}(e^{t\H \dx^2} u(t))(\xi)  \ ,\ 
\widetilde{w}(t,\xi) := \mathcal{F}(e^{-it\dx^2} w(t))(\xi)  \,,
\end{equation}
where the Fourier transform is taken only in the space variable, the equations \eqref{BO} and \eqref{eqw:intro}
essentially become 
\begin{align}
\label{dtwtu}
&\dt \widetilde{u}(\xi) = \int_{\R} e^{it\Omega(\xi,\xi_1, \xi-\xi_1)} \,\xi \, \wt{u}(\xi_1) \wt{u}(\xi-\xi_1)d\xi_1  
 \ ,\\ 
 \label{dtwtw}
 &\dt \widetilde{w} (\xi) =  \int_{\R} e^{it\Omega(\xi,\xi_1, \xi-\xi_1)} 
    \sigma(\xi,\xi_1,\xi-\xi_1) \,\frac{\xi(\xi-\xi_1)}{\xi_1} \,\wt{w}(\xi_1) \wt{u}(\xi-\xi_1)d\xi_1 \ ,
\end{align}
Here, $\Omega(\xi,\xi_1,\xi_2):= \xi|\xi| - \xi_1|\xi_1| - \xi_2|\xi_2|$ 
is the resonance relation for the BO equation 
and $\sigma(\xi,\xi_1,\xi_2)$ 
gathers the symbols of the frequency projections in the main nonlinearity of \eqref{eqw:intro} 
(see \eqref{defn:Omega}-\eqref{defn:sigma} below). 
Also, for the sake of exposition, we dropped the contribution of the negligible term of \eqref{eqw:intro}. 
We then integrate by parts in the Duhamel formulation of \eqref{dtwtw} and we obtain
\begin{equation}
\label{dtwtw:step1intro}
\begin{split}
\wt{w}(t)- \wt{w}(0) =&-2 \Bigg[ \int_{\R} \frac{e^{it'\Omega(\xi,\xi_1, \xi-\xi_1)}}{i\Omega(\xi,\xi_1,\xi-\xi_1)} 
    \sigma(\xi,\xi_1,\xi-\xi_1) \,\frac{\xi(\xi-\xi_1)}{\xi_1} \,\wt{w}(\xi_1) \wt{u}(\xi-\xi_1)d\xi_1 \Bigg]_{t'=0}^{t'=t} \\
    &\ +2 \int_0^t  \int_{\R} \frac{e^{it'\Omega(\xi,\xi_1, \xi-\xi_1)}}{i\Omega(\xi,\xi_1,\xi-\xi_1)} 
    \sigma(\xi,\xi_1,\xi-\xi_1) \,\frac{\xi(\xi-\xi_1)}{\xi_1} \, \partial_{t'} \big(\wt{w}(\xi_1) \wt{u}(\xi-\xi_1) \big) \,d\xi_1 \,dt' \,.
\end{split}
\end{equation}
While the boundary terms are fairly easy to estimate in the $H^s$-norm, $s\ge 0$, the latter term is still unfavourable. Nonetheless, due to the sign restrictions given by $\sigma(\xi,\xi_1,\xi_2)$, 
the resonance relation can be factorized, i.e. 
$\Omega(\xi,\xi_1,\xi-\xi_1) = 2\xi(\xi-\xi_1)$, and thus
\begin{equation}
\frac{2}{\Omega(\xi,\xi_1,\xi-\xi_1)} \frac{\xi(\xi-\xi_1)}{\xi_1}  = \frac{1}{\xi_1} \,.
\end{equation}
After substituting $\partial_{t'}\wt w$  and $\partial_{t'}\wt u$ 
with \eqref{dtwtw} and \eqref{dtwtu} in the last term of \eqref{dtwtw:step1intro}, and 
after undoing the change of variables \eqref{VdP}, we can write the obtained renormalized equation essentially as 
\begin{equation} \label{BOto cubicNLS}
\begin{split}
\dt w - i\dx^2w &=\dt \mathcal{N}_0^{(1)}(w,u) -i \underbrace{ \Pphi \Big( \Pphi \big( \dx^{-1} wP_{-}\dx u\big) P_{-}u\Big)}_{=:\mathcal{N}_1^{(2)}(w,u,u)}
 -i \underbrace{\Pphi\Big(\dx^{-1}w P_{-}\dx \big( u^2 \big) \Big)}_{=:\mathcal{N}_2^{(2)}(w,u,u)}
\end{split}
\end{equation}
It turns out that since we are morally dealing with cubic nonlinearities, 
the desired $H^s$-estimates can be proven only for $s>\frac14$. Hence, we proceed with a further iteration of normal form reductions, namely we apply the same strategy of integration by parts in time as above, 
now for the terms $ \mathcal{N}_1^{(2)}(w,u,u) $ and $ \mathcal{N}_2^{(2)}(w,u,u) $. 
While the first is easy to handle, the latter is  more involved due to the indefinite sign of the resonance relation.

The new ingredient in this scheme that allows us to obtain nonlinear estimates below the regularity threshold $s=\frac16$ of the result in \cite{KishimotoBO} is the use of a refined Strichartz estimate in the spirit of \cite{BCajm99,Tatajm00,BGTajm04,KochTzvetkovIMRN03,KenigKoenig}. Such estimate is obtained by applying the classical Strichartz estimate on small time intervals depending on the size of the frequency of the solution. We also refer to \cite{MolinetPilodVentoIbero18,MPVjapan} for the use of this kind of estimates for the unconditional uniqueness problem, although in a different method. 







\begin{remark}
Further iterations of normal form reductions would possibly lower the regularity of the result, although we doubt that $s=0$ could be reached without an additional tool. 
\end{remark}

\begin{remark}
In the periodic setting, it is slightly easier to work with the gauge transformation \eqref{defn:gaugeTao} 
since one can assume that $u$ has vanishing mean-value to define an anti-derivative (see Section \ref{Sec:periodic} for more details).
\end{remark}

\begin{remark} To justify rigorously the two integration by parts, we perform dyadic decompositions of each  function involved in the nonlinear terms.  Since for fixed dyadic numbers, the integrals restricted  to these dyadic pieces are absolutely convergent, we can interchange the integrals in frequency and the integrals in time rigorously in the nonlinear terms and then integrate by parts. The summations over all the dyadic frequencies are performed only at the end of the argument, after the two integrations by parts.  We refer the reader to Section \ref{normalform:sec} for more details.
\end{remark}

This technique of renormalizing the nonlinearity is akin to applying Poincar\'{e}-Dulac 
normal form reductions for ordinary differential equations. We refer to  \cite{TakTsu,NakTakTsu,BabinIlyinTiti,KwonOhIMRN,ErdoganTzirakisIMRN,ErdoganTzirakisMRL,GuoKwonOh,KwonOhYoon,ErdoganTzirakisDNLShalfline,MosincatYoon} for some applications to nonlinear dispersive equations, although the list is not exhaustive.  

This method was also used for the periodic BO by Kishimoto in \cite{KishimotoBO}, where two normal form iterations were performed. Note however that Kishimoto did not work directly on the equation \eqref{eqw:intro} of $w$, 
but instead reinjected the expression of $u$ in terms of $w$ to work with the main nonlinearity in closed form in the spirit of \cite{IonescuKenigBO}.

In~\cite{Correia}, Correia implemented the infinite-iterations normal form reductions scheme, developed in 
\cite{GuoKwonOh,KishimotoProc,KwonOhIMRN,KwonOhYoon}, and  showed unconditional uniqueness of solution to BO with initial data in the weighted Sobolev space $H^s_w:= \{f\in H^s : xf(x)\in L^2, \hat{f}(0)=0\}$, for  $s>0$. 

In~\cite{FonsecaLinaresPonce, FonsecaPonce}, the initial-value problem \eqref{BO}-\eqref{inicond} was studied 
in weighted Sobolev spaces   and some  unique continuation properties have been established. 
More recently, Kenig, Ponce, and Vega \cite{KPVuniqBO} proved the unique continuation property in regular Sobolev spaces $H^s$, for $s>\frac52$.


The paper is organized as follows. In the main body of the article, we focus on the real line case. In Section $2$, we introduce the notations, prove the refined Strichartz estimates, introduce the gauge transformation of Tao and state some basic estimates for a solution $u$ and its gauge transformation $w$. Section $3$ is the heart of the paper; there we develop two normal form iterations on the equation for $w$, which allows us to prove the key estimate for the difference of the gauges $w_1$ and $w_2$, corresponding to two solutions $u_1$ and $u_2$ evolving from the same initial data. Section $4$ is devoted to the proofs of Theorem \ref{mainthm}  and Corollary \ref{maincoro} in the real line case. Finally, in Section $5$, we explain what are the main modifications of the proofs in the periodic case. 

\medskip

\section{Prerequisites}

\subsection{Notation} 
For any $T>0$, we use the short-hand notation $C_TH^s:= C([0,T]; H^s)$. 
Unless otherwise mentioned, all Lebesgue and Sobolev norms are with respect to the spatial variable. 

We recall that the Hilbert transform on $\R$ defined by 
$\displaystyle (\H f)(x) = \mathrm{p.v.} \frac{1}{\pi} \int_{\R} \frac{f(y)}{x-y}dy $ 
has the Fourier transform 
 $\widehat{\H f}(\xi) = -i \sgn(\xi) \widehat{f}(\xi)$, 
where $\sgn(0)=0$, $\sgn(\xi)=1$ for $\xi>0$, and 
$\sgn(\xi)=-1$ for $\xi<0$.  
The Riesz projection operators $P_{\pm}$ are defined via 
$$\widehat{P_{\pm} f}(\xi) = \ind_{>0}(\pm\xi) \widehat{f}(\xi) \,,$$
where $ \ind_{>0}$ and  $ \ind_{<0}$  denote the indicator functions of the intervals $(0,\infty)$ 
and $(-\infty,0)$, respectively. 
More generally, we use $\ind_{\text{``Expr''}}$ as the indicator function for the set on which \text{``Expr''} holds true. 
We know that $P_{\pm}$ are bounded on $L^p(\R)$, only for $1<p<\infty$. 
Note that we have $\H = -iP_{+} +i P_{-}\,$. 

Let $\psi$ be a smooth bump (real-valued) even function 
that is equal to $1$ on $[-1,1]$ and vanishes outside $[-2,2]$. 
For any $N\in 2^{\Z}$, 
we use the Littlewood-Paley operators:
\begin{align*}
\widehat{P_{\le N} f}(\xi) &= \psi(N^{-1}\xi) \widehat{f}(\xi)\,,\\
P_N &=P_{\le N} -P_{\le \frac{N}{2}} \,,\\
P_{>N} &= 1-P_{\le N} \,.
\end{align*}
Also, we set 
\begin{align*}
&\Plo:= P_{\le 1}\ ,\ \Phi:= 1-\Plo\ ,\ P_{\pm\text{hi}}:= P_{\pm} \Phi \,,\\
&\PLO:=P_{\le 2}\ ,\ \PHI:=1-\PLO\ ,\ P_{\pm\text{HI}} := P_{\pm} \PHI\,.
\end{align*}
We know that $\Plo, \PLO, \Phi, \PHI, P_{\pm \textup{hi}}, P_{\pm \textup{HI}}$ 
are bounded on $L^p$, for  any $1\le p\le \infty$, while 
$P_{\pm}$ are bounded on $L^p$, for any $1< p< \infty$. 
Note that we have 
$\PHI \Plo =0\ ,\ \Phi\Plo = P_{\ge -1} P_{\le 1} \ ,\ \PHI\PLO = P_{\ge 0} P_{\le 3}$, etc. 
Also, 
$\overline{P_{\mp}f} = P_{\pm} \overline{f}\,,\ 
\overline{P_{\mp\text{hi}} f} = P_{\pm\text{hi}} \overline{f}\,,\ 
\overline{P_{\mp\text{HI}} f} = P_{\pm\text{HI}} \overline{f}\,,\ 
\overline{\Plo f} = \Plo \overline{f}\,,\ 
 \overline{\PLO f} = \PLO \overline{f}\,.$

By $D^s$ and $J^s$ we denote the Fourier multiplier operators with symbols $|\xi|^s$ and 
$\jb{\xi}^s:=(1+|\xi|)^s$, respectively. 
We use $\mathcal{F}(\,\cdot\,)$ to denote the spatial Fourier transform when 
the $\widehat{\,\cdot\,}$ 
notation is impractical. 
It is also useful to employ shorthand notation when handling nonlinear expressions on the Fourier side. 
Thus, we use for example $\xi_{12}$ in place of $\xi_1 +\xi_2$, $\xi_{123}$ in place of $\xi_1 + \xi_2+\xi_3$, etc. 



\subsection{Basic estimates}

We first recall the well-known Bernstein inequalities. 

\begin{lemma}
Let $s\ge 0$ and $1\le p\le q\le \infty$. We have 
\begin{align*}
\| P_{\le N} D^s  f\|_{L^p} &\sim N^{s} \|P_{\le N} f\|_{L^p} \,,\\
\|P_{\ge N} f \|_{L^p} &\lesssim N^{-s} \|D^s P_{\ge N} f\|_{L^p} \,,\\
\|P_{\le N} f\|_{L^q} &\lesssim N^{\frac{1}{p}-\frac{1}{q}} \|P_{\le N} f\|_{L^p}\,,\\
\| D^{\pm s} P_N  f\|_{L^p} &\sim N^{\pm s} \|P_N f\|_{L^p} \,. 
\end{align*}
\end{lemma}

Due to the gauge transformation that we use (see \eqref{defn:w:v1} below), 
the estimates provided by the following lemma 
come in handy when estimating terms involving $e^{\pm iF}$.

\begin{lemma}[{\cite[Lemma~2.7]{MolinetPilodAPDE12}}]
\label{MP:lem2p7}
Let $2\le q< \infty$ and $0\le \alpha\le \frac1q$. 
Suppose $F_1, F_2$ are two real-valued functions such that $u_j:=\dx F_j\in L^2(\R)$ for $j=1,2$. Then 
\begin{equation}
\label{lem2p7:est1}
\big\|J^{\alpha}\big(e^{\pm iF_1} g)\big\|_{L^q(\R)} 
  \lesssim \big(1+\|u_1\|_{L^2(\R)}\big) \big\|J^{\alpha} g \big\|_{L^q(\R)}
\end{equation}
and 
\begin{gather}
\label{lem2p7:est2}
\begin{split}
&\big\|J^{\alpha}\big((e^{\pm iF_1} - e^{\pm iF_2}) g\big)\big\|_{L^q(\R)} \\
 &\qquad \lesssim \Big(\|u_1-u_2\|_{L^2(\R)} 
 		+ \big\|e^{ iF_1} -e^{ i F_2} \big\|_{L_x^{\infty}(\R)} \big(1+\|u_1\|_{L^2(\R)}\big) \Big) 
         \big\|J^{\alpha} g \big\|_{L^q(\R)} \,.
\end{split}
\end{gather}
\end{lemma}

\subsection{Gauge transformation} 
We use the idea of  Tao \cite{TaoBO04}, 
namely the adaptation to the Benjamin-Ono equation 
of the Cole-Hopf transformation $u\mapsto e^{-iF}$, where $F$ is a spatial antiderivative of $u$, that transforms 
the quadratic derivative Schr\"{o}dinger equation 
$$ \dt u -i \dx^2 u =\dx(u^2)$$
into the linear Schr\"{o}dinger equation. 
However, the dispersive linear part of the Benjamin-Ono equation \eqref{BO} changes sign 
between positive and negative frequencies. 
None\-theless, 
the idea is to work with  
\begin{align}
\label{defn:W:v1}
W &:=\Pphi(e^{-i F})\,,
\end{align}
at the price of dealing with an equation for $W$ which is not in closed form, 
and subsequently inverting \eqref{defn:W:v1} is more involved than simply multiplying with $e^{iF}$. 

Since we are working at low regularity, 
we employ here the construction of the gauge transformation of Burq and Planchon 
\cite{BurqPlanchonBO} that can be carried over for $u\in C_TL^2$. 
It proceeds by constructing $F=F[u]$,  a spatial antiderivative of $u$ (i.e. $\dx F=u$), 
which also satisfies 
\begin{equation}
\label{eqF}
\dt F + \H \dx^2 F  = u^2
\end{equation}
in the sense of distributions. 
Such an $F$ is uniquely determined up to an additive constant 
More precisely, we take 
\begin{equation}
\label{defnF}
F(t,x) := \int_{\R}\psi(y) \bigg( \int_y^x u(t,z) dz \bigg) dy + G(t)\,,
\end{equation}
for some smooth, compactly supported $\psi:\R\to \R$, 
with $\int_{\R} \psi(y) dy =1$, and 
where we choose 
\begin{equation}
\label{defnG}
G(t):= \int_0^t \int_{\R} \big( -\H \psi'(y) u(t',y) + \psi(y) u(t',y)^2\big) \,dy\,dt'  \,.
\end{equation}
Note that $F$ is real-valued. 

\begin{remark}
We have that 
$e^{- iF} \in L^{\infty}(\R)$, but clearly $e^{- iF}\notin L^2(\R)$.
Hence $e^{- i F}$ is a tempered distribution on $\R$ and its Fourier transform $\widehat{e^{-i F}}$ is defined via pairing with Schwartz functions. 
Provided that we stay away from the zero frequency, i.e. $|\xi|\gtrsim 1$, 
we can make sense of 
$\widehat{e^{-iF}}(\xi)$ almost everywhere.   
Indeed, since $\dx(e^{-iF}) \in L^2(\R)$, one easily verifies that 
\begin{equation}
\label{FtreiF}
 \widehat{e^{-iF}}(\xi) = \frac{1}{i\xi}\int_{\R} e^{-ix\xi} \dx(e^{-iF}) dx  \,,
\end{equation}
 for almost every $\xi \in\R$. 
Hence, by using the Littlewood-Paley projections, 
$\Phi (e^{-iF})$,  $\PHI (e^{-iF})$, $P_{\pm \textup{hi}} (e^{-iF})$,  are well-defined $L^2(\R)$-functions. 
However, due to the possible singularity at the zero frequency which is apparent in \eqref{FtreiF}, 
$P_{\pm }(e^{-iF})$ 
might not be well-defined (unless we impose additional assumptions on $u$ itself). 
We make sense of $\Plo(e^{-iF})$, $\PLO (e^{-iF})$ not via  Littlewood-Paley projections, 
but by {defining}:  
$$\Plo (e^{-iF}) := e^{-iF} - \Phi(e^{-iF})\ ,\ \PLO (e^{-iF}):= e^{-iF} - \PHI(e^{-iF}) \,.$$ 
Still, we have $\PHI \Plo (e^{-iF}) = \PHI(e^{-iF}) -  \PHI(e^{-iF}) =0$ 
and that $\dx\Plo(e^{-iF}) = \Plo \dx(e^{-iF})$. 

Similarly, for $F$ itself we do not have information about its decay at spatial infinity,  
we only know  that $\dx F=u\in H^s_x(\R)$. 
Thus, $\Phi F, \PHI F, P_{\pm \textup{hi}}F$ are well-defined, whereas $P_{\pm} F$ might not be. 
\end{remark}


\begin{remark}
If $u$ is a solution to \eqref{BO} on $[0,T]$, i.e.  
$$u(t)= e^{-t\H \dx^2} u_0 + \int_0^t e^{-(t-t')\H \dx^2} \dx(u(t')^2) dt' \,,$$
in the sense of spatial distributions, for all $t\in [0,T]$, 
then $F=F[u]$ constructed via \eqref{defnF} is a solution to 
$\dt F+\H \dx^2F = (\dx F)^2$, 
i.e. 
\begin{equation*}
F(t) = e^{-t\H \dx^2 } F_0 + \int_0^t e^{-(t-t') \H\dx^2} \big(\dx F(t')\big)^2 dt' \,,
\end{equation*}
in the sense of spatial distributions, for all $t\in [0,T]$, 
where $F_0(x):= \int_{\R} \psi(y) \int_y^x u_0(z) dz dy$. 
\end{remark}

The following is a variant of {\cite[Lemma~4.1]{MolinetPilodAPDE12}} stated for two solutions with the same initial data.

\begin{lemma}
\label{lem:diffFs}
Assume that $u_1, u_2\in C_TL^2$ are  two solutions to \eqref{BO} on $[0,T]$ 
emanating from initial data $u_{1,0}$, $u_{2,0}$ having the same low frequency part, \textit{i.e.} $P_{lo}u_{1,0}=P_{lo}u_{2,0}$. 
Let $F_1, F_2$ denote the corresponding spatial antiderivatives of $u_1, u_2$ satisfying \eqref{eqF} 
(as per the construction above).  
Then$\|F_1|_{t=0} - F_2|_{t=0}\|_{L^{\infty}} \lesssim \|u_{1,0}-u_{2,0}\|_{L^2}$ and 
\begin{align}
\|F_1-F_2\|_{C_TL^{\infty}} &\lesssim \jb{T} 
	 \left(\|u_1\|_{C_TL^2} + \|u_2\|_{C_TL^2}\right) \|u_1 - u_2\|_{C_TL^2}\,.
\end{align}
\end{lemma}

Straightforward computations give the following equation for $W$: 
\begin{equation} 
\dt W 
-i\partial_x^2W= -2 \Pphi\big[ \big(\Pphi e^{-iF}\big) \big(P_{-} \dx^2 F \big)\big] 
  - 2\Pphi\big[ \big(\Plo e^{-iF}\big) \big(P_{-} \dx u \big)\big] \,.
\label{eqW}
\end{equation}
Note that $\big(\Pmhi e^{-iF}\big) \big(P_{-} \dx^2 F \big)$ vanishes under $\Pphi$. 
Also, by Lemma~\ref{MP:lem2p7}, if $u(t)\in H^s$ then we have $W(t)\in H^{s+1}$, 
for any $0\le s\le \frac12$.  


However, as in  \cite{MolinetPeriodicBOenergysp}, 
we prefer to work at the $H^s$-level, namely we consider 
\begin{align}
\label{defn:w:v1}
w &:= \dx W 
\end{align} 
and thus the Benjamin-Ono equation becomes
\begin{align}
\label{eqw}
\dt w - i\dx^2 w & 
	= -2 \Pphi\dx\big[\dx^{-1}w \cdot P_-\dx u\big] 
	   -2\Pphi \dx \big[ \big(\Plo e^{-iF}\big) \big(P_{-} \dx u \big)\big] 
\,.
\end{align}
The difficult term on the right-hand side is the first term and note that 
 its first factor, i.e. $\dx^{-1}w$, necessarily has larger frequency than the second factor. 
 Recall that due to the restriction of $w$ to positive high frequencies, 
  $\partial_x^{-1}w $ is defined as the Fourier multiplier with  symbol $(i\xi)^{-1}$; 
  this is not the case  for $u$ itself, hence the need to carefully construct its antiderivative $F$.

The second term on the right-hand side of \eqref{eqw} is negligible in the sense that we are essentially dealing with a quadratic term involving two smooth factors. Indeed, the estimate for the difference of two such terms is straightforward and it is given by the following lemma.

\begin{lemma}[estimate for the negligible term in \eqref{eqw}]
\label{est:negl1:v1}
Let $\sigma\ge 0$, $u_1, u_2\in L^2$ and denote 
\begin{equation*}
E(f,g) :=-2 \Pphi \dx \big[ \big(\Plo f\big) \big(P_{-} \dx g \big)\big] \,.
\end{equation*} 
Then, we have 
\begin{align}
\label{est:lem2p7}
&\Big\|E\big(e^{-iF_1} , u_1\big) - E\big(e^{-iF_2}, u_2\big) \Big\|_{H^{\sigma}} 
\lesssim \|u_1\|_{L^2} \|F_1-F_2\|_{L^{\infty}} + \|u_1 -u_2\|_{L^2} \,.
\end{align}
\end{lemma}
\begin{proof}
We can insert two $\PLO$ operators, namely we have 
$$E(f,g) =-2 \PLO \Pphi \dx \big[ \big(\Plo f\big) \big(\PLO P_{-} \dx g \big)\big] \,,$$
 and thus 
 \begin{align*}
&\Big\|E\big(e^{-iF_1} , u_1\big) - E\big(e^{-iF_2}, u_2\big) \Big\|_{H^{\sigma}}  \\
&\ \lesssim \Big\| \Plo  \Big(e^{-iF_1}- e^{-iF_2}\Big) \PLO P_{-}\dx u_1 
 + \Plo \Big(e^{-iF_2}\Big) \PLO P_{-}\dx \big(u_1-u_2 \big) \Big\|_{L^2} \\
 &\ \lesssim \big\|\Plo\big(e^{-iF_1}- e^{-iF_2}\big)\big\|_{L^{\infty}} \|\PLO P_-\dx u_1\|_{L^2} 
      + \big\|\Plo\big(e^{-iF_2}\big)\big\|_{L^{\infty}} \|\PLO P_-\dx (u_1-u_2)\|_{L^2} \\
  &\ \lesssim \|F_1-F_2\|_{L^{\infty}} \|u_1\|_{L^2} + \|u_1-u_2\|_{L^2} \,.
\end{align*}

\end{proof}

\begin{remark}
Formally (i.e. for smooth solutions or for limits of smooth solutions), one can verify that \eqref{eqW}-\eqref{eqw} hold by straightforward computations. For a low-regularity $C_TL^2$-solution $u$ to \eqref{BO}, one can proceed as in \cite[Section~2]{KishimotoBO} to justify that the gauge transformation $w$ is a solution to \eqref{eqw} in the sense of distribution. 
More precisely,  for any dyadic number $N \ge 1$, we use the truncation $u_{\le N}:= P_{\le N}u$ and define its spatial antiderivative $F_{\le N}$ by 
\begin{equation*}
F_{\le N}=F(t,x) := \int_{\R}\psi(y) \bigg( \int_y^x u_{\le N}(t,z) dz \bigg) dy + G_{\le N}(t)\
\end{equation*}
where $G_{\le N}(t):= \int_0^t \int_{\R} \big( -\H \psi'(y) u_{\le N}(t',y) + \psi(y) P_{\le N}\big(u(t',y)^2\big)\big) \,dy\,dt' $.  Then, $\partial_xF_{\le N}=u_{\le N}$ and if we define $H_{\le N}:=P_{\le N}(u^2)-u_{\le N}^2$, the equation 
\begin{equation*}
\partial_t F_{\le N}+\mathcal{H}\partial_x^2F_{\le N}=u_{\le N}^2+H_{\le N}  
\end{equation*}
holds in the classical sense. It follows that $w_{\le N}:=\dx \Pphi(e^{-iF_{\le N}})$ satisfies 
\begin{equation} \label{eq:wN}
\begin{split}
\dt w_{\le N} - i\dx^2 w_{\le N} & 
	= -2 \Pphi\dx\big[\dx^{-1}w_{\le N} \cdot P_-\dx u_{\le N}\big] -i\Pphi\dx\big[e^{-i F_{\le N}}H_{\le N}\big] \\ 
	   & \quad +E(e^{-iF_{\le N}}, u_{\le N}) 
\end{split}
\end{equation}
also in the classical sense. 
Now observe from Lemma \ref{lem:diffFs} that $(u_{\le N},w_{\le N},F_{\le N}) \to (u,w,F)$ in $C_T(L^2 \times L^2 \times L^{\infty})$. Moreover, $H_{\le N} \to 0$ in $C_TH^{-1}$ as $N \to +\infty$. Therefore, it follows from the estimate (see \cite[Section~2]{KishimotoBO})
\[ \big\|\Pphi\big[\dx^{-1}w \cdot P_-\dx u\big]\|_{H^{-1}} \lesssim \|w\|_{L^2}\|u\|_{L^2} \]
and Lemma \ref{est:negl1:v1} that the right-hand side of \eqref{eq:wN} converges to the right-hand side of \eqref{eqw} in $C_TH^{-2}$, which justifies that $w$ solves \eqref{eqw} in the distributional sense. 
\end{remark}

\subsection{Strichartz estimates}
We recall here that for $u$ a solution to 
\begin{equation} \label{BO:lin}
\dt u +\H \dx^2 u = F \,, \ u|_{t=0}=u_0
\end{equation}
we have the classical Strichartz estimates: 
\begin{equation} \label{strich:BO}
\|u\|_{L_t^pL_x^q} \lesssim \|u_0\|_{L_x^2}  + \|F\|_{L_t^1L_x^2} \,,
\end{equation}
for any Strichartz admissible pair, i.e.  $\frac2p+\frac1q =\frac12$ with $4\le p\le \infty$, $2 \le q \le +\infty$. 
Next, we follow an argument of Koch-Tzvetkov in \cite{KochTzvetkovIMRN03} and Kenig-Koenig in \cite{KenigKoenig}  for the Benjamin-Ono equation of decomposing the time interval $[0,T]$ into small subintervals whose length depends on the size of the frequency of the solution.  See also Burq-G\'erard-Tzvetkov \cite{BGTajm04} for the nonlinear Schr\"{o}dinger equation on compact manifolds, and Bahouri-Chemin \cite{BCajm99}  and Tataru \cite{Tatajm00} for the wave equation. 

\begin{lemma}[refined Strichartz estimates]
\label{lem:refStrichartz}
Let $0\le s\le \frac14$, $N\in 2^{\Z}$, $N\ge 2^6$, and $T>0$.  
We  assume that $(p,q)$ is a Strichartz admissible pair and 
we denote
\begin{equation}
\label{defn:alphasp}
\alpha(s,p): =\frac{\frac32-s}{p} - s\,.
\end{equation}

\noindent
\textup{(i)} If $u$ is a solution to \eqref{BO:lin} with 
$F=\partial_x(u_1u_2 + u_3u_4)$, then we have 
\begin{equation}
\label{refinedStrichartz}
 \|P_N u\|_{L_T^pL_x^q} \lesssim 
 T^{\frac1p} N^{\alpha(s,p)} \Big( \|P_N u\|_{L_T^{\infty}H_x^s}  
 + \|u_1\|_{L_T^{\infty}H_x^s}  \|u_2\|_{L_T^{\infty}H_x^s}  
  + \|u_3\|_{L_T^{\infty}H_x^s}  \|u_4\|_{L_T^{\infty}H_x^s} \Big) \,;
\end{equation}

\noindent
\textup{(ii)} If $w$ is a solution to 
$$\dt w - i\dx^2 w = -2\Pphi \dx \big[\dx^{-1}w_1 \cdot P_{-}\dx w_2 + \dx^{-1}w_3 \cdot P_{-}\dx w_4\big] 
  +\phi\,,$$ 
where $\supp(\wh{w_1}), \supp(\wh{w_3}) \subset(2^{-3},\infty)$ and $\supp(\wh{\phi})\subset [-2^4,2^4]$, 
then we have
\begin{equation}
\label{refinedStrichartzw}
 \|P_N w\|_{L_T^pL_x^q} \lesssim 
 T^{\frac1p} N^{\alpha(s,p)} \Big( \|P_N w\|_{L_T^{\infty}H_x^s}  
 + \|w_1\|_{L_T^{\infty}H_x^s}  \|w_2\|_{L_T^{\infty}H_x^s}  + \|w_3\|_{L_T^{\infty}H_x^s}  \|w_4\|_{L_T^{\infty}H_x^s} \Big) \,.
\end{equation}
\end{lemma}

\begin{remark}
Note that $\alpha(s,p)\searrow -s$ as $p\nearrow \infty$, but for $p=\infty$ 
we can use directly the trivial estimate
$$ \|P_N u\|_{L_T^{\infty}L_x^2} 
\sim  N^{-s} \|P_N u\|_{L_T^{\infty}H_x^s}  \,.$$
The advantage of using this refinement of the Strichartz  estimate (i.e. \eqref{refinedStrichartz}) 
is evident 
when comparing it with the  estimate
\begin{equation}
\|P_N u\|_{L_T^pL_x^q} \lesssim 
 T^{\frac1p} N^{\frac2p-s}  \|P_N u\|_{L_T^{\infty}H_x^s}  \, ,
 \end{equation}
 which follows directly from the third Bernstein inequality and H\"{o}lder inequality in time. 
\end{remark}

\begin{remark} 
Since we would also like to apply these estimates for differences of two solutions, 
we stated the estimates \eqref{refinedStrichartz} and \eqref{refinedStrichartzw} with the term involving $u_3, u_4$ and $w_3,w_4$, respectively. 
\end{remark}

\begin{proof}[Proof of Lemma~\ref{lem:refStrichartz}]
With $\delta>0$ to be chosen later, 
let $I_j = :[a_j, b_j]$ be such that $\bigcup_j I_j = [0,T]$, $b_j-a_j\sim N^{-\delta}$, and 
the number of such intervals is $\sim TN^{\delta}$. 
For (i), by using \eqref{strich:BO} we get
\begin{align*}
\|P_N u\|^p_{L_T^pL_x^q}& = \sum_j \int_{a_j}^{b_j} \|P_N u\|_{L_x^q}^p dt  \\
&\lesssim  TN^{\delta} \|P_N u\|_{L_T^{\infty}L_x^2}^p + \sum_j |I_j|^{p-1}  \|P_N F\|_{L_{I_j}^pL_x^2}^p
\end{align*}
which gives us 
\begin{align*}
\|P_N u\|_{L_T^pL_x^q} &\lesssim 
 T^{\frac1p} N^{\frac{\delta}{p}} \|P_N u\|_{L_T^{\infty}L_x^2}  
 + N^{-\big( 1- \frac{1}{p}\big)\delta} \|P_NF\|_{L_T^p L_x^2} \\
 & \lesssim 
 T^{\frac1p} N^{\frac{\delta}{p}} \|P_N u\|_{L_T^{\infty}L_x^2}  
 +T^{\frac1p} N^{-\big( 1- \frac{1}{p}\big)\delta} \|P_NF\|_{L_T^{\infty}L_x^2} \,.
\end{align*}


\noindent
In particular, for 
\begin{equation*}
F=\dx(u_1 u_2 +u_3u_4)\,,
\end{equation*}
 we get
\begin{equation*}
\|P_N u\|_{L_T^pL_x^q} \lesssim 
 T^{\frac1p} N^{\frac{\delta}{p}-s} \|P_N u\|_{L_T^{\infty}H_x^s}  
 +T^{\frac1p} N^{1-\big( 1- \frac{1}{p}\big)\delta} \Big(\|P_N(u_1u_2)\|_{L_T^{\infty}L_x^2} 
  +  \|P_N(u_3u_4)\|_{L_T^{\infty}L_x^2}  \Big)\,.
\end{equation*}
Together with 
\begin{equation}
\label{est:PNprodus}
 \|P_N(u_1u_2)\|_{L_x^2} \lesssim  N^{\frac1r - \frac12} \|u_1u_2\|_{L_x^r}
   \le N^{\frac1r - \frac12}  \|u_1\|_{L_x^{2r}} \|u_2\|_{L_x^{2r}} 
 \lesssim N^{\frac1r - \frac12} \|u_1\|_{H_x^s} \|u_2\|_{H_x^s} \,,
\end{equation}
where $1\le r\le 2$ is determined by $s=\frac12 - \frac{1}{2r}$ (or equivalently, $r= \frac{1}{1-2s}$), 
and with the same estimate for $P_N(u_3 u_4)$, 
we obtain
\begin{align*}
\|P_N u\|_{L_T^pL_x^q} \lesssim 
& T^{\frac1p} N^{\frac{\delta}{p}-s} \|P_N u\|_{L_T^{\infty}H_x^s}  \\
 &\qquad  +T^{\frac1p} N^{\frac32 -\big( 1- \frac{1}{p}\big)\delta-2s} \Big( 
 \|u_1\|_{L_T^{\infty}H_x^s}  \|u_2\|_{L_T^{\infty}H_x^s}
  + \|u_3\|_{L_T^{\infty}H_x^s}  \|u_4\|_{L_T^{\infty}H_x^s}\Big) \,.
\end{align*}
Note that the restriction on $r$ restricts us to $0\le s\le \frac14$. 
We choose $\delta$ such that 
$$\frac{\delta}{p}-s = \frac32 -\big( 1- \frac{1}{p}\big)\delta-2s \,,$$ 
or equivalently $\delta = \frac32-s$, 
and with 
$\alpha(s,p):= \frac{3}{2p} - \big(1+\frac1p\big)s$
we obtain 
\eqref{refinedStrichartz}. 

To obtain \eqref{refinedStrichartzw} we argue similarly. Indeed, by using the classical Strichartz estimate 
for the linear Schr\"{o}dinger equation, we obtain as above
\begin{equation}
\label{est:firstPNwineq}
\|P_N w\|_{L_T^pL_x^q} \lesssim 
   T^{\frac1p} N^{\frac{\delta}{p}} \|P_N w\|_{L_T^{\infty}L_x^2}  
 +T^{\frac1p} N^{-\big( 1- \frac{1}{p}\big)\delta} \|P_N G\|_{L_T^{\infty}L_x^2} \,,
\end{equation}
where we set
\begin{equation*}
G =  -2\Pphi \dx \big[\dx^{-1}w_1 \cdot P_{-}\dx w_2 + \dx^{-1}w_3 \cdot P_{-}\dx w_4\big] +\phi \,.
\end{equation*}
Note that $P_N\phi = 0$ and thus with $g_j = \mathcal{F}^{-1}\big( \big| \mathcal{F}(w_j) \big| \big)$, 
$j=1,2,3,4$, we have
\begin{equation*}
\|P_NG\|_{L_x^2} \lesssim N \|P_N(g_1 g_2)\|_{L_x^2}  + N \|P_N(g_3 g_4)\|_{L_x^2} \,.
\end{equation*}
Then, similarly to \eqref{est:PNprodus}, we get
\begin{equation*}
\|P_N(g_j g_k)\|_{L_x^2}  \lesssim N^{\frac12-2s} \|g_j\|_{H_x^s} \|g_k\|_{H_x^s} 
 =  N^{\frac12-2s} \|w_j\|_{H_x^s} \|w_k\|_{H_x^s}  \,,
\end{equation*}
for $(j,k)=(1,2)$ or $(j,k)=(3,4)$. 
From \eqref{est:firstPNwineq}, we obtain 
\begin{align*}
\|P_N w\|_{L_T^pL_x^q} &\lesssim 
 T^{\frac1p} N^{\frac{\delta}{p}-s} \|P_N w\|_{L_T^{\infty}H_x^s} \\
&\qquad  +T^{\frac1p} N^{\frac32 -\big( 1- \frac{1}{p}\big)\delta-2s} \big(
 \|w_1\|_{L_T^{\infty}H_x^s}  \|w_2\|_{L_T^{\infty}H_x^s}   +
  \|w_1\|_{L_T^{\infty}H_x^s}  \|w_2\|_{L_T^{\infty}H_x^s}   \big)
\end{align*}
and thus \eqref{refinedStrichartzw} follows by choosing $\delta$ as above.

\end{proof}

\begin{lemma}
\label{lem:lemmaA}
Let  $0<s\le \frac14$, $2\le q\le 4$ such that $\left( \frac32-s\right) \left(\frac14 -\frac{1}{2q}\right) - s<0$, and $N\in 2^{\Z}$, $N\ge 2^6$. 
Suppose $u, u^{\dagger}$ are two solutions of \eqref{BO}. 
\textup{(i)} We have 
\begin{equation}
\label{est:PNdtu_diff}
\big\| P_N \dt \big(\wt{u}- \wt{u}^{\dagger} \big)\big\|_{L_T^1L_x^2}  \lesssim 
  T N^{\frac{2}{q} +\frac12} \big(1+ \|u\|_{L_T^{\infty}H^s} + \|u^{\dagger}\|_{L_T^{\infty}H^s} \big)^3
   \|u-u^{\dagger}\|_{L_T^{\infty}H^s}\,,
\end{equation}
where $\wt{u}(t) = e^{t\H \dx^2} u(t)$ and $\wt{u}^{\dagger}(t) = e^{t\H \dx^2} u^{\dagger}(t)$.\\ 
\textup{(ii)} If $w, w^{\dagger}$ are the corresponding gauge transformations of $u, u^{\dagger}$, 
we also have
\begin{gather}
\begin{split}
\label{est:PNdtw_diff}
& \big\| P_N \dt \big(\wt{w}- \wt{w}^{\dagger} \big)\big\|_{L_T^1L_x^2} \\
&\qquad  \lesssim 
  T N^{\frac{2}{q} +\frac12}  \big(1+ \|u\|_{L_T^{\infty}H^s} + \|u^{\dagger}\|_{L_T^{\infty}H^s} \big)^6
 \big( \|w-w^{\dagger}\|_{L_T^{\infty}H^s}+ \|u-u^{\dagger}\|_{L_T^{\infty}H^s} \big)\,,
\end{split}
\end{gather} 
where $\wt{w}(t) = e^{-it \dx^2} w(t)$ and $\wt{w}^{\dagger}(t) = e^{-it \dx^2} w^{\dagger}(t)$.  
\end{lemma}
\begin{proof}
We first prove \eqref{est:PNdtu_diff} for which we set $v:= u- u^{\dagger}$
The equation on the Fourier side satisfied by $\wt{v} :=  e^{t\H\dx^2} v = \wt{u} - \wt{u}^{\dagger}$  can be rewritten as
\begin{align*}
\dt \wh{\wt{v}}(t,\xi) 
 &= i \xi \int_{\xi_{12}=\xi} e^{it\Omega(\xi,\xi_1,\xi_2)}
    \big(\wh{\wt{u}}(t,\xi_1)  + \wh{\wt{u}^{\dagger}}(t,\xi_1)  \big)\wh{\wt{v}}(t,\xi_2) d\xi_1 \\
& =i \xi e^{it|\xi|\xi} \int_{\xi_{12}=\xi}   \big( \wh{u}(t,\xi_1)  + \wh{u^{\dagger}}(t,\xi_1)  \big) \wh{v}(t,\xi_2) d\xi_1 
\,.
\end{align*}
Then, by using the H\"{o}lder and Bernstein inequalities, it follows that 
\begin{gather*}
\begin{split}
\big\| P_N \dt \wt{v} \big\|_{L_T^1L_x^2} & \lesssim N \big\|P_N\big( (u +u^{\dagger})v \big) \big\|_{L_T^1L_x^2} 
 \lesssim T^{1-\frac{2}{p}} N^{\frac{2}{q} + \frac12} \big\|P_N\big( (u+u^{\dagger}) v \big)\big\|_{L_T^{\frac{p}{2}}L_x^{\frac{q}{2}}}\\
 &\lesssim  T^{1-\frac{2}{p}} N^{\frac{2}{q} + \frac12} 
   \big( \|u\|_{L_T^pL_x^q}  +  \|u^{\dagger}\|_{L_T^pL_x^q} \big) \|v\|_{L_T^pL_x^q} 
\end{split}
\end{gather*}
where $(p,q)$ is a Strichartz admissible pair, with $p$ such that $\alpha(s,p)$ is negative, 
or equivalently   $\left( \frac32-s\right) \left(\frac14 -\frac{1}{2q}\right) - s<0$. 
Note that the refined Strichartz estimates given in Lemma~\ref{lem:refStrichartz} (i) 
apply to both $u$ and $v$. Hence we obtain
\begin{gather*}
\begin{split}
\big\| P_N \dt \wt{v} \big\|_{L_T^1L_x^2} &\lesssim T N^{\frac{2}{q} + \frac12} 
 \big(1+ \|u\|_{L_T^{\infty}H^s} + \|u^{\dagger}\|_{L_T^{\infty}H^s} \big)^2 \big(1+ \|v\|_{L_T^{\infty}H^s}\big)
  \|v\|_{L_T^{\infty}H^s}\\
 &\le T N^{\frac{2}{q} + \frac12} 
 \big(1+ \|u\|_{L_T^{\infty}H^s} + \|u^{\dagger}\|_{L_T^{\infty}H^s}\big)^3\|u-u^{\dagger}\|_{L_T^{\infty}H^s}
 \,.
\end{split}
\end{gather*}
$\quad$ 
For \eqref{est:PNdtw_diff}, we set $z= w-w^{\dagger}$, $\wt{z} = e^{-it\dx^2}z = \wt{w}-\wt{w}^{\dagger}$ and $\widetilde{E}=e^{-it\dx^2}E$. 
We have
\begin{align*}
\dt \wh{\wt{z}}(t,\xi) 
 &= i \xi \int_{\xi_{12}=\xi } e^{it\Omega(\xi,\xi_1,\xi_2)} \sigma(\xi,\xi_1,\xi_2) 
   \frac{\xi_2}{\xi_1} \big( \wh{\wt{z}}(\xi_1) \wh{\wt{u}}(\xi_2)  
     + \wh{\wt{w}^{\dagger}}(\xi_1) \wh{\wt{v}}(\xi_2)\big) d\xi_1\\
  &\qquad  + \wh{\wt{E}}\big(e^{-iF} , u\big) - \wh{\wt{E}}\big(e^{-iF^{\dagger}}, u^{\dagger}\big) \\
  &= i \xi e^{it|\xi|\xi}  \int_{\xi_{12}=\xi }   \sigma(\xi,\xi_1,\xi_2) 
   \frac{\xi_2}{\xi_1} \big( \wh{z}(\xi_1) \wh{u}(\xi_2)  + \wh{w^{\dagger}}(\xi_1) \wh{v}(\xi_2)\big) d\xi_1 \\
   &\qquad  + \wh{\wt{E}}\big(e^{-iF} , u\big) - \wh{\wt{E}}\big(e^{-iF^{\dagger}}, u^{\dagger}\big) \,,
\end{align*}
where $F$ and $F^{\dagger}$ denote the corresponding spatial antiderivatives of $u$ and $u^{\dagger}$, respectively, constructed as in the previous subsection, and $\sigma(\xi,\xi_1,\xi_2)=  \chi_+(\xi)  \wt{\chi}_+(\xi_1)  \ind_{<0}(\xi_2).$ 
Since $N$ is large enough, the $E$-terms vanish under $P_N$. 
By taking into account that $|\xi_2|<\xi_1$ on the support of $\sigma$ 
and then arguing similarly as in part (i),  
we get that $\big\| P_N \dt \wt{z} \big\|_{L_T^1L_x^2}$ is controlled by 
\begin{eqnarray*}
&N\big\| P_NP_{+hi}\big(P_{+hi}\partial_x^{-1}zP_-\partial_xu\big)\|_{L^1_TL^2_x}
+N\big\| P_NP_{+hi}\big(P_{+hi}\partial_x^{-1}w^{\dagger}P_-\partial_xv\big)\|_{L^1_TL^2_x} 
\end{eqnarray*}
We only deal with the first term as the second one is estimated similarly. By taking into account that $|\xi_2|<\xi_1$ on the support of $\sigma$, performing dyadic decompositions and using H\"older's inequality, we control the first term by
\begin{align*}
   N \sum_{N_1} &\sum_{N_2 \lesssim N_1} \| P_N\Pphi \big(\Pphi\partial_x^{-1}P_{N_1}zP_-\partial_xP_{N_2}u\big)\|_{L^1_TL^2_x} \\ 
  & \lesssim  N^{\frac2q+\frac12}T^{1-\frac2p} \sum_{N_1} \sum_{N_2 
   \lesssim N_1} \|\Pphi\partial_x^{-1}P_{N_1}z\|_{L^p_TL^q_x} \|P_-\partial_xP_{N_2}u\|_{L^p_TL^q_x} \\ 
 &\lesssim N^{\frac2q+\frac12}T^{1-\frac2p} \sum_{N_1}  \sum_{N_2 \lesssim N_1} \frac{N_2}{N_1} \|P_{N_1}z\|_{L^p_TL^q_x} \|P_{N_2}u\|_{L^p_TL^q_x} \,.
 \end{align*}
We conclude the proof by using Lemma \ref{lem:refStrichartz} as above 
(the dyadic summations are finite since $\alpha(s,p)<0$) and Lemma~\ref{MP:lem2p7}.
\end{proof}

\subsection{Estimates for solutions to the original BO in terms of gauge transformations}

Here we follow the idea from  \cite[Section~4]{KishimotoBO} to establish a control  
for  $\|u_1-u_2\|_{C_TH^s}$ in terms of $\|w_1-w_2\|_{C_TH^s}$, where 
$u_1, u_2$ are two solutions to \eqref{BO} and $w_1, w_2$ are the corresponding gauge transformations.

\begin{lemma}
\label{lem:diffinus}
Let $0\le s<\frac12$, $N\in 2^{\Z_+}$, and $T>0$. 
Assume that $u_1,u_2$ are two solutions to \eqref{BO} on $[0,T]$ with the same initial data $u_0\in H^s$
and let $w_1, w_2$ be the corresponding gauge transformations of $u_1, u_2$, respectively. 
Then, we have
\begin{align}
\label{diffinus:est1}
\|P_{\le N} (u_1-u_2)\|_{C_TH^s} &\lesssim 
 TN^{\frac32+s}\big(1+ \|u_1\|_{C_TH^s} + \|u_2\|_{C_TH^s} \big)^2 \|u_1-u_2\|_{C_TL^2} \,,\\
 \begin{split}
 \label{diffinus:est2}
 \|P_{> N} (u_1-u_2)\|_{C_TH^s}
 &  \lesssim  
   \big(1+ \|u_1\|_{C_TH^s} + \|u_2\|_{C_TH^s} \big)^4\\
  &\qquad  \times  \Big( \|w_1-w_2\|_{C_TH^s}  
+  \jb{T} \big(N^{s-\frac12} +   \|P_{>\frac{N}{2}}w_2\|_{C_TH^s} \big) 
     \|u_1-u_2\|_{C_TH^s} \Big)
     \,.
 \end{split}
\end{align}
\end{lemma}
\begin{proof}
For the low-frequency part, we use 
directly \eqref{BO:duhamel} and take the difference term by term, namely we use 
\begin{align*}
u_1-u_2 = \int_0^t e^{-(t-t')\H\dx^2} \dx\big(u_1^2-u_2^2\big)(t')\,dt' \,.
\end{align*}
By using 
the Bernstein and H\"{o}lder inequalities, we get
\begin{align*}
\|P_{\le N} (u_1-u_2)\|_{C_TH^s} 
 &\le \int_0^T \big\|P_{\le N} \dx (u_1^2-u_2^2) \big\|_{C_TH^s} dt'\\
&\lesssim TN^{\frac32+s}  \big\|P_{\le N} \big((u_1+u_2)(u_1-u_2)\big) \big\|_{C_TL^1}\\
&\lesssim TN^{\frac32+s} \big( \|u_1\|_{C_TL^2} +  \|u_2\|_{C_TL^2}\big)  \|u_1-u_2\|_{C_TL^2} \,.
\end{align*}

For the high-frequency part, we recall that since $u_1,u_2$ are real-valued, 
\begin{align*}
\|P_{>N}(u_1-u_2)\|_{C_TH^s} \sim \|P_{>N}P_+ (u_1-u_2)\|_{C_TH^s} \,.
\end{align*}
We write $u_j$ in terms of $F_j$ and $w_j$ in the following way: 
\begin{align*}
u_j & = e^{iF_j} e^{-iF_j} u_j 
    =i e^{iF_j} \dx \big[ \Pphi(e^{-iF_j}) + \Plo(e^{-iF_j})  + \Pmhi(e^{-iF_j}) \big]\\
 &=  i e^{iF_j} w_j  
    +i e^{iF_j} \Plo\dx(e^{-iF_j})  +i e^{iF_j}  \Pmhi \dx (e^{-iF_j}) \,.
\end{align*}
Therefore, we have 
\begin{align}
\nonumber
& \|P_{>N}P_+ (u_1-u_2)\|_{H^s} \\
\label{u1-u2:t1}
&\quad \leq \big\|P_{>N}P_+ \big( e^{iF_1}(w_1-w_2)\big) \big\|_{H^s} \\
\label{u1-u2:t2}
 &\qquad  + \big\|P_{>N}P_+ \big( (e^{iF_1}-e^{iF_2})w_2)\big) \big\|_{H^s} \\
\label{u1-u2:t3}
 &\qquad + \big\|P_{>N}P_+ \big( e^{iF_1} \Plo\dx (e^{-iF_1}-e^{-iF_2})\big) \big\|_{H^s} \\
\label{u1-u2:t4}
 &\qquad + \big\|P_{>N}P_+ \big( (e^{iF_1}-e^{iF_2}) \Plo\dx (e^{-iF_2})\big) \big\|_{H^s} \\
\label{u1-u2:t5}
  &\qquad + \big\|P_{>N}P_+ \big( e^{iF_1} \Pmhi\dx (e^{-iF_1}-e^{-iF_2})\big) \big\|_{H^s} \\
 \label{u1-u2:t6}
  &\qquad + \big\|P_{>N}P_+ \big( (e^{iF_1}-e^{iF_2}) \Pmhi\dx (e^{-iF_2})\big) \big\|_{H^s} \,.
\end{align}
Before estimating each term \eqref{u1-u2:t1}-\eqref{u1-u2:t6} one by one,  notice that 
by Lemma~\ref{MP:lem2p7}, 
for any $0\le \sigma\le s$ we have
\begin{align}
\nonumber
\big\| \dx \big( e^{-iF_1}- e^{-iF_2}\big)\big\|_{H^{\sigma}} &\le 
 \big\|e^{-iF_1}(u_1-u_2)\big\|_{H^{\sigma}} + 
 \big\| \big(e^{-iF_1} - e^{-iF_2}\big)u_2\big\|_{H^{\sigma}} \\
 & \lesssim  \big(1+ \|u_1\|_{H^{\sigma}} + \|u_2\|_{H^{\sigma}} \big)^2 \big( \|u_1-u_2\|_{H^{\sigma}}  +   \|F_1-F_2\|_{L^{\infty}} \big) \,.
\label{est:diffeiFs}
\end{align}

\noindent
By \eqref{lem2p7:est1},  we have 
\begin{align*}
\eqref{u1-u2:t1} \lesssim \big\|J^s\big( e^{iF_1}(w_1-w_2)\big) \big\|_{L^2}
  \lesssim \big(1+\|u_1\|_{L^2}\big)\|w_1-w_2\|_{H^s} .
\end{align*}
For the second term, we split $w_2= P_{\le \frac{N}{2}}w_2 + P_{> \frac{N}{2}}w_2$ and then we use 
Berstein's inequality, Plancherel's identity, 
\eqref{lem2p7:est2} and \eqref{est:diffeiFs} with $\sigma=0$: 
\begin{align*}
\eqref{u1-u2:t2} &\lesssim  \big\|P_{> \frac{N}{2}} J^s\big(e^{iF_1}-e^{iF_2}\big) \big\|_{L^{\infty}} 
   \big\| P_{\le \frac{N}{2}}w_2\big\|_{L^2} 
    + \big\|\big(e^{iF_1}-e^{iF_2} \big) P_{> \frac{N}{2}}w_2 \big\|_{H^s} \\
  &\lesssim N^{-(\frac12-s)} \|\dx (e^{iF_1}-e^{iF_2})\|_{L^2} \|w_2\|_{L^2}\\
 &\qquad  + \big( \|u_1-u_2\|_{L^2} + \big(1+ \|u_1\|_{H^s} + \|u_2\|_{H^s} \big)\|F_1-F_2\|_{L^{\infty}} \big)
    \big\|P_{>\frac{N}{2}} w_2\big\|_{H^s}\\
 &\lesssim \big(1+ \|u_1\|_{H^s} + \|u_2\|_{H^s} \big) \big( \|u_1-u_2\|_{L_x^2} + \|F_1-F_2\|_{L^{\infty}} \big)
  \Big(N^{s-\frac12} +  \big\|P_{>\frac{N}{2}} w_2\big\|_{H^s} \Big)\,.
\end{align*}
For the next term we can insert for free $P_{>\frac{N}{2}}P_+$ in the first factor, namely
we have 
\begin{align*}
\eqref{u1-u2:t3} &=  
 \Big\| P_{>N}P_+ \Big( P_{>\frac{N}{2}} P_+(e^{iF_1})\cdot 
    \Plo\dx \big(e^{-iF_1}-e^{-iF_2}\big)\Big) \Big\|_{H^s} \\
 &\lesssim \big\|P_{>\frac{N}{2}} P_+ J^s(e^{iF_1}) \big\|_{L^2} 
  \big\|\Plo\dx (e^{-iF_1}-e^{-iF_2}) \big\|_{L^{\infty}}\\
  &\lesssim N^{s-1}  \big\|P_{>\frac{N}{2}} P_+\dx (e^{iF_1}) \big\|_{L^2} 
   \big\| \Plo \dx(e^{-iF_1} -e^{-iF_2}) \big\|_{L^2}\\
  &\lesssim   N^{s-1} \|u_1\|_{L^2}  \big(\|u_1-u_2\|_{L^2} +\big(\|u_1\|_{H^s} + \|u_2\|_{H^s} \big) \|F_1-F_2\|_{L^{\infty}} \big) \,.
\end{align*}

\noindent
Similarly, 
we have 
\begin{align*}
\eqref{u1-u2:t4} &=
  \Big\|P_{>N}P_+ \Big(P_{>\frac{N}{2}} P_+ (e^{iF_1}-e^{iF_2}) 
   \cdot \Plo\dx (e^{-iF_2})\Big) \Big\|_{H^s}\\
 &\lesssim  N^{s-1} \|u_1\|_{L^2} \big(\|u_1-u_2\|_{L^2} 
   +\big(\|u_1\|_{H^s} + \|u_2\|_{H^s} \big) \|F_1-F_2\|_{L^{\infty}} \big) 
\end{align*}
and 
\begin{align*}
\eqref{u1-u2:t5} &\lesssim \big\|P_{>\frac{N}{2}}  P_+J^s(e^{iF_1}) \big\|_{L^{\infty}} 
  \big\|\Pmhi\dx (e^{-iF_1}-e^{-iF_2}) \big\|_{L^2}\\
  &\lesssim  N^{s-\frac12} \|u_1\|_{L^2}
   \big(  \|u_1-u_2\|_{H^s}   + \big(\|u_1\|_{H^s} + \|u_2\|_{H^s} \big) \|F_1-F_2\|_{L^{\infty}} \big) \,,
\end{align*}
where in the last step we have used \eqref{est:diffeiFs} with $\sigma=s$. 
Lastly, we argue  similarly to estimating \eqref{u1-u2:t5} by using  \eqref{est:diffeiFs} with $\sigma=0$ and we obtain: 
\begin{align*}
\eqref{u1-u2:t6} &\lesssim
 \Big\|P_{>N}P_+ \Big(P_{>\frac{N}{2}} P_+ J^s (e^{iF_1}-e^{iF_2}) 
  \cdot \Pmhi\dx (e^{-iF_2})\Big) \Big\|_{L^2}\\
&\lesssim N^{s-\frac12} \big\|\dx(e^{iF_1}-e^{iF_2})\big\|_{L^2} \|e^{-iF_2} u_2\|_{L^2}\\
  &\le   N^{s-\frac12}  \big(1+ \|u_1\|_{L^2} + \|u_2\|_{L^2} \big)^3
    \big( \|u_1-u_2\|_{L^2} +\big(\|u_1\|_{H^s} + \|u_2\|_{H^s} \big)  \|F_1-F_2\|_{L^{\infty}}\big) \,.
\end{align*}
Hence, \eqref{diffinus:est2} follows from the above estimates and Lemma~\ref{lem:diffFs}. 
\end{proof}


\section{Normal form reductions} \label{normalform:sec}
{The goal of this section is to prove an estimate for the difference of two solutions 
$w_1, w_2$ to \eqref{eqw} in terms of the difference of the corresponding solutions $u_1,u_2$ to the original equation \eqref{BO}. We proceed by renormalizing the main nonlinear term of \eqref{eqw} which introduces new nonlinearities. 
We prove multilinear estimates for these new terms in several lemmata below, which together imply the 
following proposition. 

\begin{proposition}
\label{prop:sect3}
Let $3-\sqrt{33/4}< s\le \frac14$, $T>0$, and $M>1$. 
Assume that  $u_1,u_2$ are two solutions to \eqref{BO} on $[0,T]$ with the same initial data $u_0\in H^s$. 
Then, for the  corresponding gauge transformations $w_1, w_2$, we have 
\begin{gather}
\begin{split}
\label{pfmainthm:interm}
&\| w_1- w_2\|_{C_TH^s}\\
& \ \lesssim
    \big( T M^{\frac32} + M^{-\frac{1}{16}} \big)  
    \big(1+\|u_1\|_{C_{T}H^s} + \|u_2\|_{C_{T}H^s}\big)^{10}
    \big( \| w_1- w_2\|_{C_TH^s} + \| u_1- u_2\|_{C_TH^s} \big)
     \, .
\end{split}
\end{gather}
\end{proposition}

Recall that after the gauge transformation $u\mapsto w =\dx P_{+\text{hi}} (e^{-iF})$, 
BO transforms  into 
\begin{align}
\label{eqw-1}
\dt w - i\dx^2 w & 
	= -2\Pphi\dx\big[\dx^{-1}w \cdot P_-\dx u\big] + E(e^{-iF},u)
\end{align}
(see \eqref{eqw}), where the second term, given by 
$E(e^{-iF},u) = -2\Pphi \dx \big[(\Plo e^{-iF})  (P_{-}\dx u)\big]$, 
is easy to handle via Lemma~\ref{est:negl1:v1}. 
For simplicity of writing we drop the functional arguments for this negligible term, 
i.e. we set $E: =E(e^{-iF},u)$. 
Here we  use the following change of variables:  
\begin{align}
\label{defn:wtw}
\wt w(t) &:= e^{-it\dx^2} w(t ) \,, \\
\label{defn:wtu}
\wt u(t) &:= e^{t\H\dx^2}u(t) \,,\\
\label{defn:wtE}
\wt E(t) &= e^{-it\dx^2} E(t) \,.
\end{align}
Then 
$\wt w$ satisfies pointwise in time the following equation in $H^s_x$
\begin{equation}
\label{eqwtildehat}
\wt w(t) - \wt w(0) = \int_0^t {\mathbf{N}^{(1)}}(\wt w, \wt u)(t')dt' \,,
\end{equation}
where the nonlinearity is defined  by
\begin{equation}
\begin{split}
\label{defn:bfN1}
 \mathcal{F} \big(\mathbf{N}^{(1)}(\wt w, \wt u) \big)(t,\xi) 
 &= -2i    \int_{\xi_{12}=\xi} e^{it\Omega(\xi,\xi_1,\xi_2)} \frac{\xi \xi_2}{\xi_1} 
   \sigma(\xi,\xi_1,\xi_2) \wh{\wt w}(t,\xi_1) \wh{\wt u}(t,\xi_2) d\xi_1 \\
   &\qquad  +  \widehat{\wt E}(t,\xi)  \,.
\end{split}
\end{equation}


In \eqref{defn:bfN1} above we have set  
\begin{align}
\label{defn:Omega}
\Omega(\xi,\xi_1,\xi_2) &:= \omega(\xi)-\omega(\xi_1)-\omega(\xi_2)  = \xi|\xi| - \xi_1 |\xi_1| -\xi_2|\xi_2| \,,\\
\label{defn:sigma}
\sigma(\xi,\xi_1,\xi_2) &:=  \chi_+(\xi)  \wt{\chi}_+(\xi_1)  \ind_{<0}(\xi_2) \,,
\end{align}
where
\begin{equation}
\chi_+(\xi) := (1-\psi(\xi))\ind_{>0}(\xi)
\end{equation}
is the symbol of $\Pphi$. Also, we inserted $\wt{\chi}_+(\xi_1)$, where 
$\wt{\chi}_+$ is a smooth function equal to $1$ on the support of $\chi_+$ and vanishing on a neighborhood of zero. 
Since $\chi_+$ and $\wt{\chi}_+$ play the same role (they indicate positive frequencies away from zero) we make a slight abuse of notation and replace $\wt{\chi}_+$ by $\chi_+$ 
in every occurrence below. 

Due to the sign restrictions on the frequencies $\xi,\xi_1,\xi_2$, 
we have the following factorization on the convolution plane $\xi=\xi_1+\xi_2$: 
\begin{equation}
\label{fact:Omega}
\Omega(\xi,\xi_1,\xi_2) 
= 2\xi\xi_2 \,.
\end{equation}
We note that the phase \eqref{fact:Omega} is signed, namely $\Omega(\xi,\xi_1,\xi_2) <0$. 
Also, 
we have 
\begin{equation}
\label{xi<xi1}
\jb{\xi}\sim |\xi|=\xi =\xi_1 +\xi_2<\xi_1=|\xi_1| \sim \jb{\xi_1} 
\end{equation}
and
\begin{equation}
\label{xi2<xi1}
|\xi_2|=-\xi_2 = \xi_1-\xi < \xi_1 \,. 
\end{equation}
We also rewrite here \eqref{BO:duhamel} on the Fourier side, namely we have 
\begin{align}
\label{equtildehat}
 \wh{\wt u}(t,\xi) - \wh{u_0}(\xi) 
 &= i \xi \int_0^t  \int_{\xi_{12}=\xi} e^{it'\Omega(\xi,\xi_1,\xi_2)}  \wh{\wt u}(t',\xi_1) \wh{\wt u}(t',\xi_2) d\xi_1 dt'\,,
\end{align}
with $\Omega(\xi,\xi_1,\xi_2)$ as in \eqref{defn:Omega} 
(there is no  factorization 
since there is no additional information on the signs of the frequencies involved). 

\begin{lemma}
\label{lem:C1nessint}
Let $u\in C_TL_x^2$ be a solution to \eqref{BO}, let $w$ be the corresponding gauge transformation, 
and let $\wt{u}, \wt{w}$ be given by \eqref{defn:wtu}, \eqref{defn:wtw}, respectively. 
Then for fixed $\xi$, the functions 
$$ :t\mapsto   \wh{u}(t, \xi) ,\   \wh{w}(t, \xi) ,\ \wh{\wt{u}}(t, \xi) , \ \wh{\wt{w}}(t, \xi) $$
are continuously differentiable. 
\end{lemma}
\begin{proof} We argue here for the last function, the others follow analogously. 
We claim that 
\begin{align*}
:t\mapsto \int_{\xi=\xi_{12}} e^{it\Omega(\xi,\xi_1,\xi_2)} \frac{\xi_2}{\xi_1} \sigma(\xi,\xi_1,\xi_2) 
 \wh{\wt{w}}(t,\xi_1) \wh{\wt{u}}(\xi_2)d\xi_1
\end{align*}
is continuous, which combined to \eqref{eqwtildehat}-\eqref{defn:bfN1} proves that $:t \mapsto \wh{\wt{w}}(t, \xi)$ is continuously differentiable. 
Indeed, 
since $:t\mapsto w(t,\cdot) \in L^2$ and $:t\mapsto u(t,\cdot) \in L^2$ are continuous,
the claim follows from Lebesgue's dominated convergence theorem and  the estimate 
\begin{equation}
\int_{\xi=\xi_{12}} \frac{|\xi_2|}{\xi_1} \sigma(\xi,\xi_1,\xi_2) |\wh{\wt{w}}(t,\xi_1)| 
   |\wh{\wt{u}}(t,\xi_2)|  d\xi_1  \lesssim \|w(t,\cdot)\|_{L_x^2}  \|u(t,\cdot)\|_{L_x^2} 
 \,.
\end{equation}
\end{proof}

\subsection{First step}
Let us consider the main term in \eqref{eqwtildehat}. We denote by ${\mathcal{N}^{(1)}}$ the bilinear operator given by: 
\begin{align}
 \label{defn:Ncal}
 \mathcal{F} \big({\mathcal{N}^{(1)}}(\wt w,\wt u) \big)(t,\xi) =  
 -2i    \int_{\xi=\xi_{12}} e^{it\Omega(\xi,\xi_1,\xi_2)} \frac{\xi \xi_2}{\xi_1} 
    \chi_+(\xi)  \chi_+(\xi_1)  \ind_{<0}(\xi_2) \wh{\wt w}(t,\xi_1) \wh{\wt u}(t,\xi_2) d\xi_1 \,.
\end{align}

Note that the difference between  ${\mathbf{N}^{(1)}}(\wt w, \wt u)$ 
and 
${\mathcal{N}^{(1)}}(\wt w, \wt u)$ 
is the negligible term $\wt E$.  
Next, we split 
 \begin{equation}
 \label{splitNcal}
{\mathcal{N}^{(1)} }=   {\mathcal{N}^{(1)}_{\le M}} +  {\mathcal{N}^{(1)}_{>M}} \,,
 \end{equation}
 where the two terms on the right-hand side are defined similarly to \eqref{defn:Ncal}, 
 with the multiplier including the indicator function for 
 $|\Omega(\xi,\xi_1,\xi_2)|\le M$ and $|\Omega(\xi,\xi_1,\xi_2)|> M$, respectively. 
 
 \begin{remark}
 \label{rmk:capitalVj}
 We prove the estimates in multilinear form since in the end we need an estimate for the difference of two solutions. 
 Thus we use $v_1$, $v_2$ in place of $\wt{w}(t)$ and $\wt{u}(t)$. 
 Also, for the proofs we find it useful to introduce here the notation: 
 \begin{equation}
 \label{defn:capitalVj}
 V_j:=  \mathcal{F}^{-1} \big( \big|\mathcal{F} (v_j) \big|\big) \,.
 \end{equation}
 Note that it follows from Plancherel's identity that $\|V_j\|_{H^s} = \|v_j\|_{H^s}$ for any $s\in\R$.  
 \end{remark}

\begin{lemma}
\label{lem:1stresonant}
Let $s\ge 0$ and $0<\delta<\frac12$. 
We have the following estimate pointwise in time: 
\begin{equation*}
\big\|{\mathcal{N}^{(1)}_{\le M} }(v_1, v_2) \big\|_{H^{s+\delta}} \lesssim M \|v_1\|_{H^s} \|v_2\|_{L^2} \,.
\end{equation*}
\end{lemma}
\begin{proof}
By Plancherel and  \eqref{xi<xi1}, we have 
\begin{align*}
\big\| {\mathcal{N}^{(1)}_{\le M} }(v_1, v_2) \big\|_{H^{s+\delta}} &\lesssim 
  M \bigg\| \int_{\xi=\xi_{12}} \jb{\xi_1}^{s+\delta-1} 
   |\wh{v_1}(\xi_1) \wh{v_2}(\xi_2) | \, d\xi_1 \bigg\|_{L^2_{\xi}}
  \sim M \big\| (J^{s+\delta-1} V_1) V_2\big\|_{L^2_x} \,,
\end{align*}
Then by H\"{o}lder and Sobolev inequalities, together with Plancherel,  we get 
\begin{align*}
\big\| (J^{s+\delta-1} V_1) V_2\big\|_{L^2_x} \le \big\|J^{s+\delta-1} V_1\big\|_{L^{\infty}} \big\| V_2\big\|_{L^2} \lesssim 
\big\| J^s V_1\|_{L^2} \big\| V_2\|_{L^2} \sim \| v_1 \|_{H^s}   \| v_2 \|_{L^2} \,.
\end{align*}
\end{proof}


Since we do not have a satisfactory estimate for the term 
${\mathcal{N}_{>M}}(\wt w,\wt u)$, we proceed with an integration by parts step in the temporal variable, namely 
\begin{gather}
\begin{split}
\label{IbP:1ststep}
&\int_0^t \mathcal{F}\big({\mathcal{N}^{(1)}_{>M}}(\wt w,\wt u)\big)(t',\xi) dt'  \\
&\quad =\, -2i \bigg[\int_{\xi = \xi_{12}} \frac{e^{it'\Omega(\xi,\xi_1,\xi_2)}}{\Omega(\xi,\xi_1,\xi_2) } 
  \frac{\xi \xi_2}{\xi_1} 
\ind_{|\Omega|>M}
    \chi_+(\xi)  \chi_+(\xi_1)  \ind_{\xi_2<0} \wh{\wt w}(t',\xi_1) \wh{\wt u}(t',\xi_2)  d\xi_1
     \bigg]_{t'=0}^{t'=t}\\
&\qquad +2i  \int_0^t \int_{\xi = \xi_{12}} \frac{e^{it'\Omega(\xi,\xi_1,\xi_2)}}{\Omega(\xi,\xi_1,\xi_2) }\frac{\xi \xi_2}{ \xi_1} 
\ind_{|\Omega|>M}
    \chi_+(\xi)  \chi_+(\xi_1)  \ind_{\xi_2<0} \big(\partial_{t'}\wh{\wt w}(t',\xi_1)\big) \wh{\wt u}(t',\xi_2) 
     \, d\xi_1 dt'\\
 &\qquad +2i  \int_0^t \int_{\xi = \xi_{12}} \frac{e^{it'\Omega(\xi,\xi_1,\xi_2)}}{\Omega(\xi,\xi_1,\xi_2) } \frac{\xi \xi_2}{ \xi_1} 
\ind_{|\Omega|>M}
    \chi_+(\xi)  \chi_+(\xi_1)  \ind_{\xi_2<0} \wh{\wt w}(t',\xi_1) \big(\partial_{t'}\wh{\wt u}(t',\xi_2) \big)
     \, d\xi_1 dt' \,.
\end{split}
\end{gather}
Notice that we interchanged the time integral with the frequency convolution integral several times.  
One can rigorously justify this step by using  Fubini's theorem and dyadic decomposition.
\label{NFR1:justrmk1}
Indeed, let us decompose 
\begin{equation}
\label{dyadicdecomp}
\wt{w} = \sum_{N_1} \wt{w}_{N_1} \text{ and } \wt{u} = \sum_{N_2} \wt{u}_{N_2} \,,
\end{equation} 
where 
$\wt{w}_{N_1} := P_{N_1}\wt{w}$ and $\wt{u}_{N_2}:= P_{N_2} \wt{u}$. 
We denote 
\begin{align}
\label{I1stNFR}
I^{(1)}(t,\xi):= -2i\xi
 \int_0^t \int_{\xi=\xi_{12}} e^{2i t' \xi\xi_2} \frac{\xi_2}{\xi_1}  \ind_{|\Omega|>M}  
 \chi_+(\xi)  \chi_+(\xi_1)  \ind_{\xi_2<0}  \wh{\wt w}(t',\xi_1) \wh{\wt u}(t',\xi_2)  d\xi_1 dt'
\end{align}
and
\begin{equation}
 \label{Ipiece1stNFR}
I^{(1)}_{N_1,N_2}(t,\xi) := -2i\xi
 \int_0^t  \int_{\xi=\xi_{12}}  e^{2i t' \xi\xi_2} \frac{\xi_2}{\xi_1}  \ind_{|\Omega|>M}  
 \chi_+(\xi)  \chi_+(\xi_1)  \ind_{\xi_2<0}  \wh{\wt w_{N_1}}(t',\xi_1) \wh{\wt u_{N_2}}(t',\xi_2)  d\xi_1 dt' 
\end{equation}
We now fix $\xi\in\R$ and  $0<t<T$.  
Since 
\begin{align}
\nonumber
& \int_0^t \int_{\xi=\xi_{12}}  \frac{|\xi_2|}{\xi_1}  \ind_{|\Omega|>M}  
 \chi_+(\xi)  \chi_+(\xi_1)  \ind_{\xi_2<0}  |\wh{\wt w}(t',\xi_1)| |\wh{\wt u}(t',\xi_2)|  d\xi_1 dt' 
 \le T \|w\|_{L_T^{\infty}L_x^2} \|u\|_{L_T^{\infty}L_x^2} 
 \,,\\
\nonumber
&\int_0^t  \int_{\xi=\xi_{12}}  \frac{|\xi_2|}{\xi_1}  \ind_{|\Omega|>M}  
 \chi_+(\xi)  \chi_+(\xi_1)  \ind_{\xi_2<0}  |\wh{\wt w_{N_1}}(t',\xi_1)| |\wh{\wt u_{N_2}}(t',\xi_2)|  d\xi_1 dt' \\
 \label{justrmk1:est1}
&\qquad  \le T \frac{N_2}{N_1} \|w_{N_1}\|_{L_T^{\infty}L_x^2} \|u_{N_2}\|_{L_T^{\infty}L_x^2} \,,
\end{align}
the integrals \eqref{I1stNFR} and \eqref{Ipiece1stNFR} are absolutely convergent 
and thus we can switch the order of integration defining them. 
Moreover, for any $\theta>0$, 
\begin{align}
\nonumber
\sum_{N_1}\sum_{N_2} |I^{(1)}_{N_1,N_2}| &\lesssim T \sum_{N_1}\sum_{N_2\lesssim N_1}
    \frac{N_2}{N_1}  \|w_{N_1}\|_{L_T^{\infty}L_x^2} \|u_{N_2}\|_{L_T^{\infty}L_x^2} 
    \lesssim T  \sum_{N_1} N_1^{-\theta} \|w\|_{L_T^{\infty}H_x^{\theta}} \|u\|_{L_T^{\infty}L_x^2} \\
 \label{justrmk1:est2}
 &\lesssim T  \|w\|_{L_T^{\infty}H_x^{\theta}} \|u\|_{L_T^{\infty}L_x^2} \,.
\end{align}
Therefore, by using once again Fubini's theorem and then integration by parts in time we have 
\begin{align}
&\int_0^t  \mathcal{F}\big({\mathcal{N}^{(1)}_{>M}}(\wt w,\wt u)\big)(t',\xi) dt'  = I^{(1)}
  =\sum_{N_1}\sum_{N_2}  I^{(1)}_{N_1,N_2} \label{Fubini::} \\
  &\ = -2i\xi 
   \sum_{N_1}\sum_{N_2} \int_{\xi=\xi_{12}}\int_0^t e^{2i t' \xi\xi_2} \frac{\xi_2}{\xi_1}  \ind_{|\Omega|>M}  
 \chi_+(\xi)  \chi_+(\xi_1)  \ind_{\xi_2<0}  \wh{\wt w_{N_1}}(t',\xi_1) \wh{\wt u_{N_2}}(t',\xi_2)   dt' d\xi_1\nonumber \\
 &\ = -i  \sum_{N_1}\sum_{N_2} 
  \Big[\int_{\xi=\xi_{12}} \frac{e^{2i t' \xi\xi_2}}{\xi_1}  \ind_{|\Omega|>M}  
 \chi_+(\xi)  \chi_+(\xi_1)  \ind_{\xi_2<0}  \wh{\wt w_{N_1}}(t',\xi_1) 
  \wh{\wt u_{N_2}}(t',\xi_2)d\xi_1 \Big]_{t'=0}^{t'=t} \nonumber \\
 &\ \quad +i  \sum_{N_1}\sum_{N_2} \int_{\xi=\xi_{12}} \int_0^t 
   \frac{e^{2i t' \xi\xi_2}}{\xi_1}  \ind_{|\Omega|>M}  \chi_+(\xi)  \chi_+(\xi_1)  \ind_{\xi_2<0} 
    \partial_{t'} \big( \wh{\wt w_{N_1}}(t',\xi_1) \wh{\wt u_{N_2}}(t',\xi_2)\big) dt' d\xi_1 \nonumber
\end{align} 
The splitting into two terms in the last step above is justified as 
the first summation-integral is absolutely convergent 
(the estimate is similar to \eqref{justrmk1:est1} and \eqref{justrmk1:est2} above). 
Thus, 
\begin{align*}
&\sum_{N_1}\sum_{N_2} 
  \Big[\int_{\xi=\xi_{12}} \frac{e^{2i t' \xi\xi_2}}{\xi_1}  \ind_{|\Omega|>M}  
 \chi_+(\xi)  \chi_+(\xi_1)  \ind_{\xi_2<0}  \wh{\wt w_{N_1}}(t',\xi_1) 
  \wh{\wt u_{N_2}}(t',\xi_2)d\xi_1 \Big]_{t'=0}^{t'=t}\\
 &\ = \Big[\int_{\xi=\xi_{12}} \frac{e^{2i t' \xi\xi_2}}{\xi_1}  \ind_{|\Omega|>M}  
 \chi_+(\xi)  \chi_+(\xi_1)  \ind_{\xi_2<0}  \wh{\wt w}(t',\xi_1) 
  \wh{\wt u}(t',\xi_2)d\xi_1 \Big]_{t'=0}^{t'=t}\,.
\end{align*}
For the second summation-integral, due to Lemma~\ref{lem:C1nessint} 
we can use the product rule to distribute $\partial_{t'}$ to the two factors. 
However, we cannot ensure summation in $N_1$. 
\begin{remark}
\label{rmk:infinitesumN1}
Indeed, the first resulting term (i.e. when $\partial_{t'}$ falls on $\wh{\wt w_{N_1}}(t',\xi_1)$) would be
\begin{align}
&\sum_{N_1}\sum_{N_2}  \int_{\xi=\xi_{12}} \int_0^t 
   \frac{e^{2i t' \xi\xi_2}}{\xi_1}  \ind_{|\Omega|>M}  \chi_+(\xi)  \chi_+(\xi_1)  \ind_{\xi_2<0} 
    \psi_{N_1}(\xi_1) \partial_{t'} \big( \wh{\wt w}(t',\xi_1)\big) \wh{\wt u_{N_2}}(t',\xi_2) dt' d\xi_1 \,,
\end{align}
and observe that by using Lemma~\ref{lem:lemmaA} (ii) (with $s,q$ satisfying the hypothesis), 
we have 
\begin{align*}
&\sum_{N_1}\sum_{N_2}  \int_0^t \int_{\xi=\xi_{12}} 
  \frac{1}{\xi_1}  \ind_{|\Omega|>M}  \chi_+(\xi)  \chi_+(\xi_1)  \ind_{\xi_2<0} 
   \psi_{N_1}(\xi_1) |\partial_{t'} \big( \wh{\wt w}(t',\xi_1)\big)|  |\wh{\wt u_{N_2}}(t',\xi_2)|  d\xi_1 dt' \\
&\lesssim \sum_{N_1}\sum_{N_2\lesssim N_1} \frac{1}{N_1} \|P_{N_1}\dt
  \wt{w}\|_{L_T^1L_x^2} \|\wt{u}_{N_2}\|_{L_T^{\infty}L_x^2}
 \\ & \lesssim T \sum_{N_1} N_1^{\frac2q-\frac12}\big(1+\|w\|_{L_T^{\infty}H_x^s}\big)^6\big( \|u\|_{L_T^{\infty}H_x^s}+\|w\|_{L_T^{\infty}H_x^s}\big)
  \|u\|_{L_T^{\infty}H_x^s}
  \,.
\end{align*}
Moreover, it is worth to note that under the assumptions of Lemma~\ref{lem:lemmaA}, one cannot have $\frac2q-\frac12<0$.
\end{remark}

For the convenience of writing let us introduce notation\footnote{Note that to furthermore simplify the writing we drop the explicit temporal variable  except in the factor $e^{it\Omega(\xi,\xi_1,\xi_2)}$ which is used for the next iteration of integrating by parts in time. 
Also, we point out that all the nonlinearities that appear below depend on $M$.} 
for the boundary terms that appear on the right-hand side of \eqref{IbP:1ststep}:
\begin{align*}
\mathcal{F}\big({\mathcal{N}^{(1)}_{0}}(\wt w,\wt u)\big)(t,\xi) 
& =-2i  \int_{\xi_{12}=\xi} \frac{e^{it\Omega(\xi,\xi_1,\xi_2)}}{\Omega(\xi,\xi_1,\xi_2)} \frac{\xi\xi_2}{\xi_1} 
\ind_{|\Omega|>M}
    \chi_+(\xi)  \chi_+(\xi_1)  \ind_{<0}(\xi_2) \wh{\wt w}(\xi_1) \wh{\wt u}(\xi_2) \big) d\xi_1 \,.
\end{align*}

Once we localize the functions $\wt w$ and $\wt u$, 
the normal form reduction \eqref{IbP:1ststep} reads
\begin{equation*}
{\mathcal{N}^{(1)}_{>M}}(\wt w_{N_1},\wt u_{N_2})
= \dt {\mathcal{N}^{(1)}_{0}}(\wt w_{N_1},\wt u_{N_2})
  - {\mathcal{N}^{(1)}_{0}}(\dt \wt w_{N_1},\wt u_{N_2})
 -  {\mathcal{N}^{(1)}_{0}}(\wt w_{N_1},\dt \wt u_{N_2}) \,.
\end{equation*}

Thus far, we \emph{formally} have
\begin{gather}
\begin{split}
\label{eq:N1bigM}
\int_0^t  \mathcal{F}\big({\mathcal{N}^{(1)}_{>M}}(\wt w,\wt u)\big)(t',\xi) dt'  =&\,  
 \Big[{\mathcal{N}^{(1)}_{0}}(\wt w,\wt u)(t') \Big]_{t'=0}^{t'=t}\\
 &\ - \sum_{N_1}\sum_{N_2} \int_0^t {\mathcal{N}^{(1)}_{0}}(\partial_{t'} \wt w_{N_1}, \wt u_{N_2})(t')dt'\\  
 &\ - \sum_{N_1}\sum_{N_2} \int_0^t {\mathcal{N}^{(1)}_{0}}( \wt w_{N_1}, \partial_{t'} \wt u_{N_2})(t') dt'
 \,.  
\end{split}
\end{gather}
In view of Remark~\ref{rmk:infinitesumN1} and since we do not need to interchange the dyadic summations with the integrals until after the last normal form transformation, 
we proceed to the next step by fixing the dyadic numbers $N_1$ and $N_2$. 
Note that once $N_1$ and $N_2$ are fixed, we can freely switch the order of the time and frequency convolution integrals and it is easy to check  the absolute convergence of all the integrals on the right-hand side 
 by using the Cauchy-Schwarz inequality and Lemma~\ref{lem:lemmaA}.

The following provides a straightforward estimate for the first boundary term. 

\begin{lemma} 
\label{lem:bdrystep1}
Let $s\ge 0$ and $\delta<\frac12$.
We have the following estimate pointwise in time: 
\begin{equation*}
\big\|\mathcal{N}^{(1)}_0(v_1, v_2) \big\|_{H^{s+\delta}}
  \lesssim 
  M^{-\frac18+\frac{\delta}{4}} 
   \|v_1\|_{H^s}
  \|v_2\|_{L^2} \,.
\end{equation*}
\end{lemma}
\begin{proof}
By \eqref{xi<xi1} and by using $M<|\Omega|<|\xi_1|^2$,  we have
\begin{align*}
\frac{\jb{\xi}^{s+\delta}}{|\xi_1|} \lesssim \jb{\xi_1}^{s+\delta-1} 
  \lesssim M^{\frac12(\delta-\frac12+\theta)}  \jb{\xi_1}^{s-\frac12-\theta} \ ,\ 0<\theta<\frac12-\delta\,.
\end{align*}
Therefore, with $V_j$ as in \eqref{defn:capitalVj}, 
and $\theta= \frac12(\frac12-\delta)$, 
we get 
\begin{align*}
\big\|\mathcal{N}^{(1)}_0(v_1, v_2) \big\|_{H^{s+\delta}} & \lesssim 
  M^{-\frac14(\frac12-\delta)} 
   \big\| (J^{s-\frac12-\theta} V_1 ) (P_- V_2)\big\|_{L^2} \\
 & \lesssim M^{-\frac18+\frac{\delta}{4}}  \big\| J^{s-\frac12-\theta} V_1 \big\|_{L^{\infty}} \|V_2 \|_{L^2} \\
  &\lesssim M^{-\frac18+\frac{\delta}{4}} \big\| v_1\big\|_{H^s} \|v_2\|_{L^2}  \,,
\end{align*}
where in the last step we have used the Sobolev embedding $H^{\frac12+\theta}\subset L^{\infty}$ and Plancherel's identity. 
\end{proof}

Next, 
the substitution of $\partial_{t'} \wt w_{N_1}$ in 
${\mathcal{N}^{(1)}_0}(\partial_{t'} \wt w_{N_1},\wt u_{N_2})(t') $ 
is easily  justified 
using again Lemma~\ref{lem:lemmaA} (ii). 
At this point, we prefer to change the label $N_1$ into $N_{12}$ to maintain a coherent notation in view of the convention that $\xi_{12}=\xi_1+\xi_2$. We also change $N_2$ into $N_3$. 
As we argued above, once the dyadic numbers are fixed, it is easy to check the absolute convergence of the 
double integrals (time and frequency). 
Thus, by using \eqref{eqwtildehat}, \eqref{defn:bfN1}, and   \eqref{fact:Omega}, 
we have
\begin{gather}
\begin{split}
\label{1stNFR:dtwsubstitution}
&\int_0^t \mathcal{F}\left({\mathcal{N}^{(1)}_{0}}(\partial_{t'} \wt w_{N_{12}},\wt u_{N_3})\right)(t',\xi) dt'  \\ 
&\ =  \int_{\eta+\xi_3=\xi} \int_0^t  \frac{e^{it'\Omega(\xi,\eta,\xi_3)}}{\eta} 
\ind_{|\Omega|>M}
    \chi_+(\xi)  \chi_+(\eta)  \ind_{<0}(\xi_3) \big( \partial_{t'} \wh{\wt w_{N_{12}}}\big) (\eta)
      \wh{\wt u_{N_3}}(\xi_3) 
    dt' d\eta  \\
 &\ =  \int_{\eta+\xi_3=\xi} \int_0^t \frac{e^{it'\Omega(\xi,\eta,\xi_3)}}{\eta} \ind_{|\Omega|>M}
    \chi_+(\xi)  \chi_+(\eta)  \ind_{<0}(\xi_3)  \psi_{N_{12}}(\eta)  \wh{\wt u_{N_3}}(\xi_3)  \\
&\qquad \times 
    \Bigg[\int_{\xi_{12}=\eta} 
   \sum_{N_1} \sum_{N_2} e^{it\Omega(\eta,\xi_1,\xi_2)}\frac{2\eta\xi_2}{i\xi_1} 
     \chi_+(\eta) \wt\chi_+(\xi_1) \ind_{<0}(\xi_2)
    \wh{{\wt w}_{N_1}}(\xi_1)  \wh{ {\wt u}_{N_2}}(\xi_2) d\xi_1 
   + \wh{E}(\eta)
   \Bigg] dt' d\eta\\
 &\ =\sum_{N_1} \sum_{N_2}  \int_0^t\int_{\xi=\xi_{123}}
 {e^{it\Omega^{(2)}_1(\xi,\xi_1,\xi_2,\xi_3) }}
\mathfrak{m}^{(2)}_1(\xi,\xi_1,\xi_2,\xi_3) \,
 \wh{\wt w}(\xi_1) \wh{\wt u}(\xi_2) \wh{\wt u}(\xi_3)  \,d\xi_1 d\xi_2 \,dt' \\
  &\qquad + \int_0^t \mathcal{F} \left( \mathcal{N}^{(1)}_{0}(\wt E_{N_{12}}, \wt u_{N_3})\right)(t',\xi) dt' \,,
\end{split}
\end{gather}
where
\begin{align*}
\Omega^{(2)}_1(\xi,\xi_1,\xi_2,\xi_3) &= \Omega(\xi,\xi_{12},\xi_3) +\Omega(\xi_{12},\xi_1,\xi_2) 
\end{align*}
and 
\begin{gather}
\begin{split}
\label{m21}
\mathfrak{m}^{(2)}_1(\xi,\xi_1,\xi_2,\xi_3) :=& -2i \frac{\xi_2}{\xi_1} 
 \,\ind_{| \Omega(\xi,\xi_{12},\xi_3)|>M}
    \chi_+(\xi)   \chi_+(\xi_1)   \chi_+(\xi_{12})^2 \ind_{<0}(\xi_2)   \ind_{<0}(\xi_3) \\
   &\quad \times \psi_{N_{12}}(\xi_{12}) \psi_{N_1}(\xi_1) \psi_{N_2}(\xi_2)   \psi_{N_3}(\xi_3) \,.
\end{split}
\end{gather}
In the last step, 
 for the split into two integrals in $t'$ and $\eta$ is justified since 
the term containing the smooth term $\wh{E}$ is an absolutely convergent integral. 
Furthermore, the changing of integration order 
of $\displaystyle\int_{\eta+\xi_3 = \xi}$, 
$\displaystyle \int_0^t$, $\displaystyle\int_{\xi_{12}=\eta}$, and $\displaystyle \sum_{N_1}\sum_{N_2}$ 
is justified as in \eqref{justrmk1:est1}--\eqref{Fubini::} and by taking into account that 
$\wh{ {\wt u}_{N_3}}(\cdot )  =\psi_{N_3}(\cdot) \wh{\wt u}(\cdot)$ is localized (and thus also in $L^1$). 
Indeed, 
 for $\xi$, $t$, $N_{12}$ and $N_3$ fixed, 
 we have by using twice the Cauchy-Schwarz inequality and Plancherel's identity 
\begin{gather}
\begin{split}
\label{1stNFRdtw:checkingabsconv}
\sum_{N_1}\sum_{N_2 \lesssim N_1}& \int_0^t \int_{\eta+\xi_3=\xi} \int_{\xi_{12}=\eta} \left| \frac{\eta \xi_2}{\eta \xi_1} \right|  
 |\wh{\wt u_{N_3}}(\xi_3)| |\wh{{\wt w}_{N_1}}(\xi_1)| |  \wh{ {\wt u}_{N_2}}(\xi_2)| d\xi_1d\eta dt' \\ & 
\lesssim \sum_{N_1}\sum_{N_2 \lesssim N_1} \frac{N_2}{N_1} \int_0^t \int_{|\eta| \sim N_{12}}  |\wh{\wt u_{N_3}}(\xi-\eta)| \|\wh{{\wt w}_{N_1}}\|_{L^2} \|\wh{ {\wt u}_{N_2}}\|_{L^2} d\eta dt \\ &
\lesssim T N_{12}^{\frac12} \sum_{N_1}\sum_{N_2 \lesssim N_1} \frac{N_2}{N_1}\|u_{N_3}\|_{L^{\infty}_TL^2_x} \| w_{N_1}\|_{L^{\infty}_TL^2_x} \|u_{N_2}\|_{L^{\infty}_TL^2_x} \\ & 
  \lesssim T N_{12}^{\frac12}\|u_{N_3}\|_{L^{\infty}_TL^2_x} \| w_{N_1}\|_{L^{\infty}_TH^{0_+}_x} \|u_{N_2}\|_{L^{\infty}_TL^2_x} <\infty .
\end{split}
\end{gather} 
Therefore the sum-integral above is absolutely convergent and we can put the summation in $N_1$, $N_2$ outside of the integrals in $\eta$ and $\xi_1$.

Let us introduce the nonlinearity $ \mathcal{N}^{(2)}_{1} $ defined by
\begin{align}
& \mathcal{F}\big(\mathcal{N}^{(2)}_{1} (v_1, v_2, v_3)\big)(t,\xi) = 
\int_{\xi=\xi_{123}} \hspace{-2mm}{e^{it\Omega^{(2)}_1(\xi,\xi_1,\xi_2,\xi_3) }}
\mathfrak{m}^{(2)}_1(\xi,\xi_1,\xi_2,\xi_3) \,
   \wh{v_1}(\xi_1) \wh{v_2}(\xi_2)  \wh{v_3}(\xi_3) \,d\xi_1 d\xi_2 
\end{align}
and we note that 
due to the 
frequency restrictions in $\mathfrak{m}^{(2)}_1$ 
we  have that   
\begin{align*}
\Omega(\xi,\xi_{12},\xi_3)<0 \ ,\ \Omega(\xi_{12},\xi_1,\xi_2) <0 \ ,
\end{align*}
\begin{equation}
\Omega^{(2)}_1(\xi,\xi_1,\xi_2,\xi_3) = 2\xi\xi_3 + 2\xi_{12} \xi_2 \ ,
\end{equation}
\begin{equation}
\label{freqineqN21}
\xi, |\xi_3| < \xi_{12}<\xi_1 \ \text{ and }\  |\xi_2|<\xi_1 \,.
\end{equation}

We  rewrite \eqref{1stNFR:dtwsubstitution} as
\begin{align*}
&\int_0^t\!{\mathcal{N}^{(1)}_{0}}(\partial_{t'} \wt w_{N_{12}},\wt u_{N_3}) dt' 
 = \sum_{N_1}\sum_{N_2} \int_0^t\!\mathcal{N}^{(2)}_{1} (\wt w, \wt u, \wt u)(t')dt' 
    +  \int_0^t\!\mathcal{N}^{(1)}_{0}(\wt E_{N_{12}}, \wt u_{N_3})(t') dt' 
 \,.
\end{align*}
In the next subsection, we will work on  
$\displaystyle \int_0^t\!\mathcal{N}^{(2)}_{1} (\wt w, \wt u, \wt u)(t')dt'$. 
We point out that the multiplier $\mathfrak{m}_1^{(2)}$ depends on 
 the dyadic numbers $N_1, N_2, N_3$, and $N_{12}$, all of which are henceforth  fixed. 

Next, we move to the last term of \eqref{eq:N1bigM}. 
By using \eqref{equtildehat} and arguing as above, we get 
(for fixed $N_1$ and $N_{23}$): 
\begin{gather}
\begin{split}
\label{1stNFR:dtusubstitution}
&\int_0^t {\mathcal{N}^{(1)}_{0}}(\wt w_{N_1}, \partial_{t'} \wt u_{N_{23}}) dt'  \\ 
&\ = -i \int_0^t  \int_{\xi_1+\eta=\xi} \frac{e^{it'\Omega(\xi,\xi_1,\eta)}}{\xi_1} 
\ind_{|\Omega|>M}
    \chi_+(\xi)  \chi_+(\xi_1)  \ind_{<0}(\eta) \wh{\wt w_{N_{1}}} (\xi_1)
     \big( \partial_{t'}  \wh{\wt u_{N_{23}}}\big)(\eta) 
    d\xi_1 dt'   \\
&\ =  \int_0^t  \int_{\xi_1+\eta=\xi} \int_{\xi_{23}=\eta} \sum_{N_2} \sum_{N_3}
 e^{it'\Omega^{(2)}_2(\xi,\xi_1,\xi_2,\xi_3)} 
 \mathfrak{m}^{(2)}_*(\xi,\xi_1,\xi_2,\xi_3) 
   \wh{\wt w} (\xi_1) \wh{\wt u}(\xi_2)  \wh{\wt u}(\xi_3) 
    d\xi_2 d\xi_1 dt'   
 \,,
\end{split}
\end{gather}
where
\begin{align*}
\Omega^{(2)}_2(\xi,\xi_1,\xi_2,\xi_3) &= \Omega(\xi,\xi_{1},\xi_{23}) +\Omega(\xi_{23},\xi_2,\xi_3) \,.
\end{align*}
and 
\begin{align}
\label{defn:mstar2}
\mathfrak{m}^{(2)}_*(\xi,\xi_1,\xi_2,\xi_3) &= -i
\frac{\xi_{23}}{\xi_1}  \,\ind_{|\Omega(\xi,\xi_{1},\xi_{23})|>M}
    \chi_+(\xi)  \chi_+(\xi_1)  \ind_{<0}(\xi_{23})    
    \psi_{N_1}(\xi_1) \psi_{N_{23}}(\xi_{23}) \psi_{N_2}(\xi_2) \psi_{N_3}(\xi_3)\,.
\end{align}
Since, for $t, \xi, N_1, N_{23}$ fixed, and for any $\theta>0$, 
\begin{gather}
\begin{split}
\label{1stNFRdtu:checkingabsconv}
& \sum_{N_2}\sum_{N_3} \int_0^t \int_{\xi_1+\eta = \xi} \int_{\xi_{23}=\eta} \left| \frac{\xi_{23}}{\xi_1} \right| 
   |\wh{{\wt w}_{N_1}}(\xi_1)| |\wh{ {\wt u}_{N_2}}(\xi_2)|  |\wh{ {\wt u}_{N_3}}(\xi-\xi_1-\xi_2)|  \, d\xi_2d\xi_1dt' \\
 &\qquad \lesssim   \frac{N_{23}}{N_1} \sum_{N_2}\sum_{N_3} \int_0^t  \int_{\xi_1+\eta = \xi} 
   |\wh{{\wt w}_{N_1}}(\xi_1)| \|\wh{ {\wt u}_{N_2}}\|_{L^2}\|\wh{ {\wt u}_{N_3}}\|_{L^2}  d\xi_1dt' \\
 &\qquad \lesssim  T N_1^{-\frac12} N_{23}
     \|\wh{ {\wt w}_{N_1}}\|_{L^2}  \|  u\|_{H^{\theta}}^2 
        \sum_{N_2}\sum_{N_3} N_2^{-\theta} N_3^{-\theta}  <\infty
      \,,
\end{split}
\end{gather}
we can bring the summation in $N_2$ and $N_3$ outside in the last term of \eqref{1stNFR:dtusubstitution}, i.e. 
\begin{gather}
\begin{split}
\label{1stNFR:dtusubstitution_2}
&\int_0^t {\mathcal{N}^{(1)}_{0}}(\wt w_{N_1}, \partial_{t'} \wt u_{N_{23}}) dt'  \\ 
&\quad =  \sum_{N_2}\sum_{N_3}   \int_0^t  \int_{\xi_1+\eta=\xi} \int_{\xi_{23}=\eta}
 e^{it'\Omega^{(2)}_2(\xi,\xi_1,\xi_2,\xi_3)}  \mathfrak{m}^{(2)}_*(\xi,\xi_1,\xi_2,\xi_3) 
   \wh{\wt w} (\xi_1) \wh{\wt u}(\xi_2)  \wh{\wt u}(\xi_3) 
     d\xi_1 d\xi_2 dt'   
 \,.
\end{split}
\end{gather}

The frequency restrictions for this term only give us 
$$ \Omega(\xi,\xi_{1},\xi_{23}) = 2\xi\xi_{23} <0$$ 
and 
\begin{equation}
\label{freqineqN22}
 \xi, |\xi_{23}| <\xi_1\,.
\end{equation}
We discuss the sign of the term $\Omega(\xi_{23},\xi_2,\xi_3)$: 
\begin{align}
\label{Omega1in2ndstep}
&\Omega(\xi_{23},\xi_2,\xi_3) = 
 -\xi_{23}^2-\xi_2|\xi_2|-  \xi_3 |\xi_3| =
\begin{cases}
-2\xi_2\xi_3 &\,,\ \text{if } \xi_2<0, \xi_3< 0\,,\\
-2\xi_2 \xi_{23} &\,,\ \text{if } \xi_2\ge0, \xi_3<0\,,\\
-2\xi_3 \xi_{23} &\,,\ \text{if } \xi_2<0, \xi_3\ge 0\,.
\end{cases}
\end{align}
Note that due to the symmetry of the multiplier, the second and third branch
 in \eqref{Omega1in2ndstep} give the same term. 
Thus we  consider the regions:
\begin{align*}
&R^{(2)}_{\le M}:= \{ |\xi_{12}|\le 1\} \cup \{ |\Omega^{(2)}_2(\xi,\xi_1,\xi_2,\xi_3)|\le M\} \,,\\
&R^{(2)}_{2}:= \{ \xi_2<0, \xi_3<0\} \setminus R^{(2)}_{\le M} \,,\\
&R^{(2)}_{3}:= \{\xi_2<0, \xi_3\ge0\}\setminus R^{(2)}_{\le M} \,,
\end{align*}
and we split the 
multiplier $\mathfrak{m}^{(2)}_*$
into three terms
\begin{align}
\mathfrak{m}^{(2)}_* = 
 \mathfrak{m}^{(2)}_* \ind_{R^{(2)}_{\le M}} +  \mathfrak{m}^{(2)}_* \ind_{R^{(2)}_2} 
 +  \mathfrak{m}^{(2)}_* \ind_{R^{(2)}_3}
  \,.
\end{align}
Note that the same estimate \eqref{1stNFRdtu:checkingabsconv} holds for each of these multipliers, and therefore after introducing the nonlinearities 
$\mathcal{N}^{(2)}_{\le M} , \mathcal{N}^{(2)}_{2}$, and $\mathcal{N}^{(2)}_3$   defined by 
\begin{align*}
&\F\big( \mathcal{N}^{(2)}_{\le M} (v_1,v_2,v_3)\big)(t,\xi) \\
&\qquad = -i \int_{\xi=\xi_{123}} 
 \hspace{-4mm}{e^{it\Omega^{(2)}_2(\xi,\xi_1,\xi_2,\xi_3) }}
\big(\mathfrak{m}^{(2)}_* \ind_{R^{(2)}_{\le M}}\big)(\xi,\xi_1,\xi_2,\xi_3) \,
  \wh{v_1}(\xi_1) \wh{v_2}(\xi_2)  \wh{v_3}(\xi_3) \,d\xi_1 d\xi_2 \,,\\
&\F\big( \mathcal{N}^{(2)}_j (v_1,v_2,v_3)\big)(t,\xi) \\
&\qquad = \int_{\xi=\xi_{123}} 
 \hspace{-4mm}{e^{it\Omega^{(2)}_2(\xi,\xi_1,\xi_2,\xi_3) }}
\big(\mathfrak{m}^{(2)}_* \ind_{R^{(2)}_j}\big)(\xi,\xi_1,\xi_2,\xi_3) \,
   \wh{v_1}(\xi_1) \wh{v_2}(\xi_2)  \wh{v_3}(\xi_3) \,d\xi_1 d\xi_2 \,, 
   \text{ for $j=2,3$,}
\end{align*}
we can rewrite \eqref{1stNFR:dtusubstitution_2} as
\begin{gather}
\begin{split}
\label{1stNFR:dtusubstitution_3}
&\int_0^t {\mathcal{N}^{(1)}_{0}}(\wt w_{N_1}, \partial_{t'} \wt u_{N_{23}}) dt'  \\ 
&\ 
= 
 \sum_{N_2}\sum_{N_3}  \int_0^t  \mathcal{N}^{(2)}_{\le M} (\wt w, \wt u, \wt u) dt' 
+ \sum_{N_2}\sum_{N_3} \int_0^t  \mathcal{N}^{(2)}_{2} (\wt w, \wt u, \wt u) dt' 
+ \sum_{N_2}\sum_{N_3} \int_0^t  \mathcal{N}^{(2)}_{3} (\wt w, \wt u, \wt u) dt' 
 \,.
\end{split}
\end{gather}
Note that in the first term we put back inner-most the summations in $N_2, N_3$. 
The first term is easy to handle in view of the following estimate:

 
\begin{lemma}
\label{lem:N2leM}
Let $s\ge 0$ and $\delta<\min\{s, \frac12\}$. 
We have the following estimate pointwise in time: 
\begin{equation*}
\big\|\mathcal{N}^{(2)}_{\le M} (v_1, v_2, v_3) \big\|_{H^{s+\delta}} 
  \lesssim M^{\frac32} \prod_{j=1}^3 \|P_{N_j}v_j\|_{H^s} \,.
\end{equation*}
\end{lemma}
\begin{proof}
 
\noindent 
If $|\xi_{12}|\le1$, we easily have $\jb{\xi_3}\sim \jb{\xi} <\jb{\xi_1}$ (see \eqref{freqineqN22}) and thus 
\begin{align*}
\big\|\mathcal{N}^{(2)}_{\le M} (v_1, v_2,v_3) \big\|_{H^{s+\delta}} 
&\lesssim \|\PLO\big( (J^sV_1) V_2 \big) (J^{\delta} V_3) \|_{L^2} 
 \le  \|\PLO\big( (J^s  V_1) V_2\big) \|_{L^{\infty}} \|J^{\delta} V_3\|_{L^2}\\
& \lesssim  \|  (J^s  V_1) V_2 \|_{L^1} \|V_3\|_{H^{\delta}} 
 \lesssim  \|v_1\|_{H^s} \|v_2\|_{L^2} \|v_3\|_{H^{\delta}}\,.
\end{align*}

\noindent
Now assume that $|\xi_{12}|>1$ and $ | \Omega^{(2)}_2(\xi,\xi_1,\xi_2,\xi_3) | \le M$, 
where we recall that 
$\Omega^{(2)}_2(\xi,\xi_{1},\xi_{2}, \xi_3) = \Omega(\xi,\xi_{1},\xi_{23}) +\Omega(\xi_{23},\xi_2,\xi_3)$. 
We recall here that  the multiplier of $\mathcal{N}^{(2)}_{\le M}$ is 
$\mathfrak{m}_*^{(2)}\ind_{R_{\le M}^{(2)}}$.
Notice that on the first branch of \eqref{Omega1in2ndstep}, 
i.e. when $\xi_2<0$ and $\xi_3<0$ the conditions  
\begin{align*} 
| \Omega^{(2)}_2(\xi,\xi_1,\xi_2,\xi_3) | \le M\  \Leftrightarrow\  |\xi \xi_{23} - \xi_2\xi_3| \le \frac{M}{2}
\end{align*}
 and 
 \begin{align*}
 |\Omega(\xi,\xi_1,\xi_{23})|>M \  \Leftrightarrow\  |\xi \xi_{23}|>\frac{M}{2} 
 \end{align*}
 cannot hold simultaneously. 
Hence it remains to discuss the third branch of \eqref{Omega1in2ndstep}, i.e. 
$\xi_2<0$ and $\xi_3\ge0$ (the second branch follows by the symmetry of the multiplier in $\xi_2, \xi_3$). 
In this case we have $\Omega^{(2)}_2(\xi,\xi_1,\xi_2,\xi_3) = 2\xi_{12}\xi_{23}$ and thus 
$| \xi_{23}| \le \frac{M}{2}$. Since on the support of $\mathfrak{m}^{(2)}_{\le M}$ 
we also have 
$\xi<\xi_1$ It follows that 
\begin{align*}
\big\|\mathcal{N}^{(2)}_{\le M} (v_1, v_2,v_3) \big\|_{H^{s+\delta}} 
&\lesssim M \big\| \big(J^{s+\delta-1}V_1 \big) P_{\le M}(V_2 V_3)\big\|_{L^2} 
\le M \big\| J^{s+\delta-1}V_1 \big\|_{L^{\infty}} \big\|P_{\le M}(V_2 V_3)\big\|_{L^2} \\
&\lesssim  M^{\frac32} \big\| J^{s}V_1 \big\|_{L^2} \big\|V_2 V_3\big\|_{L^1} 
\le M^{\frac32} \|v_1\|_{H^s} \|v_2\|_{L^2} \|v_3\|_{L^2} \,.
\end{align*}
\end{proof}

\begin{remark}
A version of the estimate above with $\delta=0$ follows analogously to 
\cite[Lemma~2.3]{KwonOhYoon} taking into account that $\big|\frac{\xi_{23}}{\xi_1} \big|<1$ 
on the support of 
$\mathfrak{m}^{(2)}_{\le M}$. 
However, here we exploit that $\big|\frac{\xi_{23}}{\xi_1} \big|\lesssim \frac{M}{\jb{\xi_1}}$ and this allows us to obtain the estimate of $\mathcal{N}^{(2)}_{\le M}$ in the $H^{s+\delta}$-norm 
(albeit at the cost of a higher power on $M$ in the right-hand side). 
\end{remark}


Formally, we summarize the normal form reductions in this first step: 
\begin{gather}
\begin{split}
\wt w(t) - \wt w_0 =&\int_0^t  \mathcal{N}^{(1)}_{\le M}(\wt w, \wt u)(t') dt' + \int_0^t \wt E(t')dt' + 
 \big[\mathcal{N}^{(1)}_{0}(\wt w, \wt u)(t')\big]_{t'=0}^{t'=t} \\
 &\ - \sum_{N_{12}}\sum_{N_3\lesssim N_{12}} \Bigg\{ 
    \sum_{N_1}\sum_{N_2\lesssim N_1} \int_0^t\!\mathcal{N}^{(2)}_{1} (\wt w, \wt u, \wt u)(t')dt' 
    +  \int_0^t\!\mathcal{N}^{(1)}_{0}(\wt E_{N_{12}}, \wt u_{N_3})(t') dt' \Bigg\} \\
 &\ -\sum_{N_1}\sum_{N_{23} \lesssim N_1}\Bigg\{
\sum_{N_2}\sum_{N_3}  \int_0^t  \mathcal{N}^{(2)}_{\le M} (\wt w, \wt u, \wt u) dt' \\
 &\qquad\qquad \qquad
+ \sum_{N_2}\sum_{N_3} \int_0^t  \mathcal{N}^{(2)}_{2} (\wt w, \wt u, \wt u) dt' 
+ \sum_{N_2}\sum_{N_3} \int_0^t  \mathcal{N}^{(2)}_{3} (\wt w, \wt u, \wt u) dt' \Bigg\}
 \end{split}
\end{gather}
Notice that 
due to Lemma~\ref{lem:bdrystep1} 
and Lemma~\ref{lem:N2leM} (applied with $s-\delta'$ and $\delta+\delta'$ for some small enough $\delta'>0$), respectively, we have 
\begin{align*}
 \sum_{N_{12}}\sum_{N_3\lesssim N_{12}} \! \int_0^t \!
 \big\|\mathcal{N}^{(1)}_{0}(\wt E_{N_{12}}, \wt u_{N_3})(t')\big\|_{H^{s+\delta}} dt'  
  & \lesssim  T M^{-\frac18+\frac{\delta+\delta'}{4}}
   \sum_{N_{12}} \sum_{N_3 \lesssim N_{12}}\!\!N_{12}^{-\delta'} \|\wt E\|_{H^s} \|\wt u\|_{L^2}\\
   & \lesssim  TM^{-\frac18+\frac{\delta+\delta'}{4}} \|\wt E\|_{H^s} \|\wt u\|_{L^2}
\end{align*}
and 
\begin{gather}
\begin{split}
\label{summationofNlesM2}
& \sum_{N_{1}}\sum_{N_{23}\lesssim N_{1}} \sum_{N_2}\sum_{N_3}  \! \int_0^t \!
 \big\| \mathcal{N}^{(2)}_{\le M} (\wt w, \wt u, \wt u) (t')\big\|_{H^{s+\delta}} dt'  \\
 &\qquad\  \lesssim  TM^{\frac32} \sum_{N_{1}}\sum_{N_{23}\lesssim N_{1}} 
    \sum_{N_2}\sum_{N_3} 
    N_1^{-\delta'} N_2^{-\delta'} N_3^{-\delta'} \|\wt w\|_{H^s}  \|\wt u\|^2_{H^s} 
     \lesssim  TM^{\frac32} \|\wt w\|_{H^s}  \|\wt u\|^2_{H^s} \,.
\end{split}
\end{gather}
Thus, 
it remains to handle the terms 
$\mathcal{N}^{(2)}_1(\wt w, \wt u, \wt u)$, $\mathcal{N}^{(2)}_{2}(\wt w, \wt u, \wt u)$,  and 
$\mathcal{N}^{(2)}_{3}(\wt w, \wt u, \wt u)$. 
These terms are all nonresonant and therefore we can proceed with a second step of  integration by parts in time. 

\subsection{Second step}

Let us recall here the terms to which we have to apply a second step of integration by parts in time, their phases and their multiplier symbols on the Fourier side: 
\begin{align*}
 \mathcal{F}\big(\mathcal{N}^{(2)}_{j}(v_1,v_2,v_3)\big)(\xi) 
&= \int_{\xi=\xi_{123}} \hspace{-3mm} e^{it\Omega^{(2)}_j(\xi,\xi_1,\xi_2,\xi_3)} 
\mathfrak{m}_j^{(2)}(\xi,\xi_1,\xi_2,\xi_3) \,
\wh{v_1}(\xi_1) \wh{v_2}(\xi_2)  \wh{v_3}(\xi_3) \,d\xi_1 d\xi_2 \,,
\end{align*}
$ j=1,2,3$, 
respectively with phases given by 
\begin{align*}
&\Omega^{(2)}_1(\xi,\xi_1,\xi_2,\xi_3) =2 \xi\xi_3 + 2\xi_{12}\xi_2 \,,\\
&\Omega^{(2)}_2(\xi,\xi_1,\xi_2,\xi_3) = 2\xi\xi_{23} -2 \xi_{2}\xi_3 \,,\\
&\Omega^{(2)}_3(\xi,\xi_1,\xi_2,\xi_3)  = 2\xi\xi_{23} -2 \xi_{3}\xi_{23}=2\xi_{12}\xi_{23} \,,
\end{align*}
and multipliers 
given by
\begin{align*}
\mathfrak{m}^{(2)}_1(\xi,\xi_1,\xi_2,\xi_3) 
  = & -2i \frac{ \xi_2}{\xi_1} \ind_{| \xi\xi_3|>\frac{M}{2}}\ind_{| \xi\xi_3 + \xi_{12}\xi_2|>\frac{M}{2}} 
\chi_+(\xi)  \chi_+(\xi_{12})^2  \chi_+(\xi_1)  \ind_{<0}(\xi_2)   \ind_{<0}(\xi_3) \\
 &\quad \times \psi_{N_{12}}(\xi_{12}) \psi_{N_1}(\xi_1) \psi_{N_2}(\xi_2)   \psi_{N_3}(\xi_3) 
\,,\\
\mathfrak{m}^{(2)}_2(\xi,\xi_1,\xi_2,\xi_3)
 =& i \frac{ \xi_{23}}{\xi_1} \ind_{| \xi\xi_{23}|>\frac{M}{2}}\ind_{|\xi\xi_{23} - \xi_{2}\xi_3|>\frac{M}{2}}  
\ind_{|\xi_{12}|>1}
\chi_+(\xi)  \chi_+(\xi_1) \ind_{<0}(\xi_2)  \ind_{<0}(\xi_3)    \\
&\quad \times \psi_{N_1}(\xi_1) \psi_{N_{23}}(\xi_{23}) \psi_{N_2}(\xi_2) \psi_{N_3}(\xi_3) 
\,, \\
\mathfrak{m}^{(2)}_3(\xi,\xi_1,\xi_2,\xi_3) 
=& i \frac{ \xi_{23}}{\xi_1}  \ind_{| \xi\xi_{23}|>\frac{M}{2}} \ind_{|\xi\xi_{23} - \xi_{3}\xi_{23}|>\frac{M}{2}} 
 \ind_{|\xi_{12}|>1} \chi_+(\xi)  \chi_+(\xi_1) \ind_{<0}(\xi_2)  \ind_{\ge0}(\xi_3)  \ind_{<0}(\xi_{23})  \\
  &\quad \times \psi_{N_1}(\xi_1) \psi_{N_{23}}(\xi_{23}) \psi_{N_2}(\xi_2) \psi_{N_3}(\xi_3) 
  \,.
\end{align*}


\noindent
Since $\wt w$ and $\wt u$ factors are all localized in frequency, 
the interchange of the time integral and frequency integrals is justified
by \eqref{1stNFRdtw:checkingabsconv} (for $j=1$) and \eqref{1stNFRdtu:checkingabsconv} (for $j=2,3$), 
and thus
after applying integration by parts in time, 
together with Lemma~\ref{lem:C1nessint}, 
we get 
\begin{align}
\label{2ndNFR:IbP}
\mathcal{N}^{(2)}_j(\wt w, \wt u, \wt u)  
 = \dt \mathcal{N}^{(2)}_{j,0}(\wt w, \wt u, \wt u) - \mathcal{N}^{(2)}_{j,0}(\dt \wt w, \wt u,  \wt u)
  - \mathcal{N}^{(2)}_{j,0}(\wt w,\dt \wt u, \wt u) - \mathcal{N}^{(2)}_{j,0}(\wt w, \wt u, \dt\wt u)\,,
\end{align}
where 
\begin{align*}
\mathcal{F}\big( \mathcal{N}^{(2)}_{j,0}(v_1,v_2, v_3) \big)(\xi) &= 
\int_{\xi=\xi_{123}} e^{it \Omega^{(2)}_j(\xi,\xi_1,\xi_2,\xi_3)} 
 \frac{\mathfrak{m}^{(2)}_j(\xi,\xi_1,\xi_2,\xi_3)}{i \Omega^{(2)}_j(\xi,\xi_1,\xi_2,\xi_3)} 
 \wh{v_1}(\xi_1) \wh{v_2}(\xi_2)  \wh{v_3}(\xi_3) d\xi_1 d\xi_2 
\,.
\end{align*}

The estimates for the boundary terms appearing in the second step are provided by the following: 

\begin{lemma}
\label{lem:bdryterms2ndstep}
Let $s\ge 0$. 
We have the following estimates pointwise in time: 
\begin{equation*}
\Big\|\mathcal{N}^{(2)}_{j,0}(v_1,v_2,v_3) \Big\|_{H^{s+\frac12}} \lesssim 
 M^{-\frac14+} 
 \|v_1\|_{H^s} \|v_2\|_{L^2}  \|v_3\|_{L^2} 
 \ ,\ j=1,2,3\,.
\end{equation*}
\end{lemma}
\begin{proof}
One checks that for each $j=1,2,3$  we have $M\lesssim \xi_1^2$. 
Let us also denote $V_k:=  \mathcal{F}^{-1} \big( \big|\mathcal{F} (v_k) \big|\big)$, $k=1,2,3$. 
If $j=1$ and $j=3$, we have 
\begin{align*}
\bigg|\frac{\mathfrak{m}^{(2)}_j(\xi,\xi_1,\xi_2,\xi_3)}{\Omega^{(2)}_j(\xi,\xi_1,\xi_2,\xi_3)} \bigg| 
 \jb{\xi}^{s+\frac12}
  \lesssim \xi_1^{-\frac12+s} \xi_{12}^{-1}\lesssim 
  M^{-\frac14+}   \jb{\xi_1}^{s-} \jb{\xi_{12}}^{-1+} \,.
\end{align*}
Then, by H\"{o}lder's inequality and Sobolev embedding,  
\begin{align*}
\Big\|\mathcal{N}^{(3)}_{j,0}(v_1,v_2,v_3) \Big\|_{H^{s+\frac12}}
 &\lesssim M^{-\frac14+} \big\|
   J^{-1+} \big( J^{s-} V_1 V_2 \big) V_3 \big\|_{L^2}\\
 &\lesssim M^{-\frac14+}\big\| J^{-1+} \big( J^{s-} V_1 V_2 \big)  \big\|_{L^{\infty}} 
      \| v_3 \big\|_{L^2}\\
 &\lesssim M^{-\frac14+} \big\| J^{s-} V_1 V_2\big\|_{L^{1+}} \| v_3 \big\|_{L^2}\\
 &\lesssim M^{-\frac14+}
   \big\| J^{s-} V_1\big\|_{L^{2+}}  \|v_2 \big\|_{L^2} \| v_3 \|_{L^2}\\
 &\lesssim M^{-\frac14+}
   \big\|  v_1\big\|_{H^s}  \|v_2 \big\|_{L^2} \| v_3 \|_{L^2}\,.
\end{align*}
If $j=2$, we have 
\begin{align*}
\bigg|\frac{\mathfrak{m}^{(2)}_2(\xi,\xi_1,\xi_2,\xi_3)}{\Omega^{(2)}_2(\xi,\xi_1,\xi_2,\xi_3)} \bigg| 
 \jb{\xi}^{s+\frac12}
  \lesssim \xi_1^{-1} \xi^{-\frac12+s} \lesssim  M^{-\frac14+}   \jb{\xi_1}^{-\frac12+s-} \jb{\xi}^{-\frac12-} \,.
 \end{align*}
Then, by H\"{o}lder's inequality and Sobolev embedding,  
\begin{align*}
\Big\|\mathcal{N}^{(3)}_{2,0}(v_1,v_2,v_3) \Big\|_{H^{s+\frac12}}
 &\lesssim M^{-\frac14+} \big\| J^{-\frac12-} \big( J^{-\frac12+s-} V_1 V_2 V_3\big)  \big\|_{L^2}\\
  &\lesssim M^{-\frac14+} \big\| J^{-\frac12+s-} V_1 V_2 V_3 \big\|_{L^{1}}\\
  &\lesssim M^{-\frac14+} \big\| J^{-\frac12+s-} V_1 \big\|_{L_x^{\infty}} \|V_2\|_{L_x^2} \|V_3\|_{L_x^2}\\
  &\lesssim M^{-\frac14+}  \big\|  v_1\big\|_{H^s}  \|v_2 \big\|_{L^2} \| v_3 \|_{L^2}\,.
\end{align*}
\end{proof}


Lastly, the following estimates together with those of Lemma~\ref{lem:lemmaA}
allow us to also handle the remaining terms on the right hand side of \eqref{2ndNFR:IbP}: 
\begin{lemma}
\label{lem:newEstNj02}
Let $0<s<\frac14$ and $\delta\ge 0$.  We have the following estimate for $j=1,2,3$ and any $\theta>0$
\begin{gather}
\begin{split}
& \sup_N \big\|P_N \mathcal{N}^{(2)}_{j,0}(P_{N_1}v_1,P_{N_2}v_2, P_{N_3}v_3) \big\|_{H^{s+\delta}} \\
 &\qquad \lesssim 
 N_{max}^{-1-s+\delta+\theta} N_{med}^s N_{min}^s 
 \|P_{N_1} v_1\|_{L^2} \|P_{N_2}v_2\|_{L^2}  \|P_{N_3}v_3\|_{L^2} \,,
\end{split}
\end{gather}
where $N_{max}, N_{med}, N_{min}$ are the maximum, the median, and the minimum of $N_1, N_2, N_3$, respectively. 
\end{lemma}
\begin{proof}
Due to Plancherel's identity, we have that 
\begin{gather}
\begin{split}
\label{estlemA}
&  \big\|P_N \mathcal{N}^{(2)}_{j,0}(P_{N_1}v_1,P_{N_2}v_2, P_{N_3}v_3) \big\|_{H^{s+\delta}} \\
&\  \lesssim N^{s+\delta} \Bigg\| \int_{\xi=\xi_{123}}
  \left| \frac{\mathfrak{m}^{(2)}_j(\xi,\xi_1,\xi_2,\xi_3)}{\Omega^{(2)}_j(\xi,\xi_1,\xi_2,\xi_3)} \right|
  \psi_{N_1}(\xi_1)|\wh{v_1}(\xi_1)|  \psi_{N_2}(\xi_2)|\wh{v_2}(\xi_2)| \psi_{N_3}(\xi_3) | \wh{v_3}(\xi_3)| 
    d\xi_1 d\xi_2   \Bigg\|_{L^2_{\xi}}
\end{split}
\end{gather}
and that the supremum in $N$ is taken over all dyadic numbers $\ge \frac12$ 
(due to the restriction on $\xi$ in $\mathfrak{m}^{(2)}_j$).  
As before, we denote $V_k:=  \mathcal{F}^{-1} \big( \big|\mathcal{F} (v_k) \big|\big)$, $k=1,2,3$. 
We discuss separately the three cases depending on $j$.\\ 
\noindent
\textbf{Case 1: $j=1$.} We recall that the multiplier 
$\frac{m^{(2)}_1(\xi,\xi_1,\xi_2,\xi_3)}{\Omega^{(2)}_1(\xi,\xi_1,\xi_2,\xi_3)}$ is 
\begin{align*}
 \frac{-i\xi_2}{\xi_1(\xi\xi_3 + \xi_{12}\xi_2)}
  \ind_{| \xi\xi_3|>\frac{M}{2}}\ind_{| \xi\xi_3 + \xi_{12}\xi_2|>\frac{M}{2}} 
\chi_+(\xi)  \chi_+(\xi_{12})^2  \chi_+(\xi_1)  \ind_{<0}(\xi_2)   \ind_{<0}(\xi_3) 
  \,.
\end{align*}
Due to the sign restrictions,
we have  $\max\{\xi,|\xi_3|\} < \xi_{12} < \xi_1$ and $|\xi_2| < \xi_1$.\\

\textbf{Subcase 1.i: $|\xi_{12}|\sim |\xi_1|$.} 
In this case, 
\begin{align*}
\Bigg| \frac{m^{(2)}_1(\xi,\xi_1,\xi_2,\xi_3)}{\Omega^{(2)}_1(\xi,\xi_1,\xi_2,\xi_3)} \Bigg| 
\lesssim \xi_1^{-2} 
\end{align*}
and thus, we deduce from Berstein's inequality that 
\begin{align*}
\textup{LHS of \eqref{estlemA}} & \lesssim N_1^{-2+s+\delta} 
 \big\|P_N\big(P_{N_1}V_1 P_{N_2}V_2 P_{N_3}V_3 \big)\big\|_{L_x^2}\\
 & \lesssim N_1^{-\frac32+s+\delta} 
 \|P_{N_1}V_1\|_{L_x^{\frac{1}{2s}}}  \|P_{N_2}V_2\|_{L_x^{\frac{2}{1-2s}}} 
   \|P_{N_3}V_3\|_{L_x^{\frac{2}{1-2s}}}\\
 & \lesssim N_1^{-1-s+\delta} N_2^s N_3^s  \prod_{j=1}^3 \|P_{N_j}v_j\|_{L_x^2}
   \,.
\end{align*} \\
\textbf{Subcase 1.ii: $|\xi_{12}|\ll |\xi_1|$ and $\xi\lesssim |\xi_3|$.}
In this case we necessarily have $|\xi_2|\sim \xi_1$. We use 
\begin{align*}
\Bigg| \frac{m^{(2)}_1(\xi,\xi_1,\xi_2,\xi_3)}{\Omega^{(2)}_1(\xi,\xi_1,\xi_2,\xi_3)} \Bigg| 
\lesssim \xi_1^{-1} \xi_{12}^{-1}
\end{align*}
to get that 
\begin{align*}
\textup{LHS of \eqref{estlemA}} & \lesssim N^{s+\delta} N_1^{-1} \sum_{K\lesssim N_1} K^{-1}  
 \Big\|P_N\Big(P_K\big( P_{N_1}V_1 P_{N_2}V_2 \big) P_{N_3}V_3 \Big)\Big\|_{L_x^2}\\
&  \lesssim N_3^{s} N_1^{-1+\delta} \sum_{K\lesssim N_1}
   \big\|P_K\big( P_{N_1}V_1 P_{N_2}V_2 \big) \big\|_{L_x^{1}} \|P_{N_3}V_3\|_{L_x^2}\\
 &  \lesssim N_1^{-1-s+\delta+\theta} N_2^s N_3^s \prod_{j=1}^3 \|P_{N_j}v_j\|_{L_x^2}
\,,
\end{align*}
for any $\theta>0$.\\
\textbf{Subcase 1.iii: $|\xi_{12}|\ll |\xi_1|$ and $\xi\gg |\xi_3|$.} 
Note that in this case $|\xi_2|\sim \xi_1$ and $\xi_{12}\sim \xi$. Then, 
we have
\begin{align*}
\Bigg| \frac{m^{(2)}_1(\xi,\xi_1,\xi_2,\xi_3)}{\Omega^{(2)}_1(\xi,\xi_1,\xi_2,\xi_3)} \Bigg| 
\lesssim \xi_1^{-1} \xi^{-\frac12} \xi_{12}^{-\frac12} \,.
\end{align*}
It follows that
\begin{align*}
\textup{LHS of \eqref{estlemA}} & \lesssim N^{-\frac12+s+\delta} N_1^{-1} \sum_{K \sim N} K^{-\frac12}
   \Big\|P_N\Big(P_K\big( P_{N_1}V_1 P_{N_2}V_2 \big) P_{N_3}V_3 \Big)\Big\|_{L_x^2}\\
  &\lesssim N^{\delta} N_1^{-1} \sum_{K\sim N} K^{-\frac12}
   \Big\|P_N\Big(P_K\big( P_{N_1}V_1 P_{N_2}V_2 \big) P_{N_3}V_3 \Big)\Big\|_{L_x^{\frac{1}{1-s}}}\\
   &\lesssim  N_1^{-1+\delta}   \sum_{K\sim N} K^{-\frac12}
    \big\|P_K\big( P_{N_1}V_1 P_{N_2}V_2 \big) \big\|_{L_x^2} \|P_{N_3}V_3\|_{L_x^{\frac{2}{1-2s}}}\\
  &\lesssim  N_1^{-1+\delta} N_3^s  \|P_{N_1}V_1 P_{N_2}V_2\|_{L_x^1} \|P_{N_3}V_3\|_{L_x^2}\\
  &\lesssim  N_1^{-1-s+\delta} N_2^s N_3^s  \prod_{j=1}^3 \|P_{N_j}v_j\|_{L_x^2}
  \,.
\end{align*}
%
%
\textbf{Case 2: $j=2$.} We recall that the multiplier 
$\frac{m^{(2)}_2(\xi,\xi_1,\xi_2,\xi_3)}{\Omega^{(2)}_2(\xi,\xi_1,\xi_2,\xi_3)}$ is 
\begin{align*}
 \frac{i\xi_{23}}{2\xi_1(\xi\xi_{23}- \xi_2\xi_3)} 
    \ind_{| \xi\xi_{23}|>\frac{M}{2}}\ind_{|\xi\xi_{23} - \xi_{2}\xi_3|>\frac{M}{2}}  
\ind_{|\xi_{12}|>1} \chi_+(\xi)  \chi_+(\xi_1) \ind_{<0}(\xi_2)  \ind_{<0}(\xi_3) 
 \,.
\end{align*}
Due to the sign restrictions, we have  $\max\{|\xi_2|,|\xi_3|\}<|\xi_{23}| <\xi_1$ and $\xi<\xi_1$. 
Without loss of generality we may assume that $|\xi_2|\le |\xi_3|$ so that $|\xi_{23}|\le 2|\xi_3|$ 
and note that 
$N_{\max} = N_1$, $N_{\textup{med}}=N_3$, and $N_{\min} =N_2$. 
Then 
\begin{align*}
\Bigg| \frac{m^{(2)}_2(\xi,\xi_1,\xi_2,\xi_3)}{\Omega^{(2)}_2(\xi,\xi_1,\xi_2,\xi_3)} \Bigg| 
\lesssim \frac{|\xi_{23}|^{\frac12-s}}{\xi_1 \xi^{\frac12+s} |\xi_2 \xi_3|^{\frac12-s}} 
\lesssim  \frac{1}{\xi_1 \xi^{\frac12+s} |\xi_2 |^{\frac12-s}} \,.
\end{align*}
\noindent
\textbf{Subcase 2.i: $\xi\sim \xi_1$ (i.e. $N\sim N_1$).} We have
\begin{align*}
\textup{LHS of \eqref{estlemA}}& \lesssim N_1^{-\frac32+\delta} N_2^{-\frac12 + s}
   \big\|P_N\big(P_{N_1}V_1 P_{N_2}V_2 P_{N_3}V_3 \big)\big\|_{L_x^2}\\
  &\lesssim N_1^{-\frac32+\delta} N_2^{-\frac12 + s}
  \|P_{N_1}V_1\|_{L_x^{\frac1s}}  \|P_{N_2}V_2\|_{L_x^{\infty}}  \|P_{N_3}V_3\|_{L_x^{\frac{2}{1-2s}}}\\
&\lesssim  N_1^{-1-s+\delta} N_2^s N_3^s  \prod_{j=1}^3 \|P_{N_j}v_j\|_{L_x^2}
\,.
\end{align*}
\textbf{Subcase 2.ii: $\xi\ll \xi_1$.} Then we necessarily have $\xi_1\sim |\xi_{23}|$. 
Since $\xi_2$ and $\xi_3$ have the same sign, 
we cannot have that $|\xi_3|\ll \xi_1$, and thus $|\xi_3|\sim \xi_1$ (or equivalently $N_3\sim N_1$).  
Then
\begin{align*}
\textup{LHS of \eqref{estlemA}}& \lesssim N^{\delta}N_1^{-1} N_2^{-\frac12 + s}
   \big\|P_N\big(P_{N_1}V_1 P_{N_2}V_2 P_{N_3}V_3 \big)\big\|_{L_x^1}\\
   &\lesssim N_1^{-1+\delta}  N_2^{-\frac12 + s}
    \|P_{N_1}V_1\|_{L_x^2}  \|P_{N_2}V_2\|_{L_x^{\infty}}  \|P_{N_3}V_3\|_{L_x^2}\\
    &\lesssim N_1^{-1-s+\delta}  N_2^{s} N_3^s \prod_{j=1}^3 \|P_{N_j}v_j\|_{L_x^2}\,.
\end{align*}
\textbf{Case 3: $j=3$.} We recall that 
$\frac{m^{(2)}_3(\xi,\xi_1,\xi_2,\xi_3)}{\Omega^{(2)}_3(\xi,\xi_1,\xi_2,\xi_3)}$ is 
\begin{align*}
\frac{i}{2\xi_1\xi_{12}} \ind_{| \xi\xi_{23}|>\frac{M}{2}} \ind_{|\xi_{12}\xi_{23}|>\frac{M}{2}}
  \ind_{|\xi_{12}|>1}   \chi_+(\xi)  \chi_+(\xi_1) \ind_{<0}(\xi_2)  \ind_{\ge0}(\xi_3)  \ind_{<0}(\xi_{23})   \,.
\end{align*}
Due to the sign restrictions, we have $\max\{|\xi_{23}|,\xi\} <\xi_1$ and $\xi_3 <|\xi_2|$. 
Clearly, it holds that 
\begin{align*}
\Bigg| \frac{m^{(2)}_{3}(\xi,\xi_1,\xi_2,\xi_3)}{\Omega^{(2)}_{3}(\xi,\xi_1,\xi_2,\xi_3)} \Bigg| 
\lesssim \jb{\xi_1}^{-1} \jb{\xi_{12}}^{-1} \,.
\end{align*}
\textbf{Subcase 3.i: $|\xi_2|\ll \xi_1$.} Observe that $\xi_1 \sim |\xi_{12}| \sim \xi $, and  
$N_{\max} =N_1$, $N_{\textup{med}} = N_2$, $N_{\min}=N_3$. Then, $\textup{LHS of \eqref{estlemA}}$ is estimated exactly as in Subcase 1.i.\\
\textbf{Subcase 3.ii: $|\xi_2|\sim  \xi_1$ and $\xi\lesssim \xi_3$.} In this case we have
$N_{\min}=N_3$ and $N_{\max}\sim N_{\textup{med}}\sim N_1\sim N_2$ 
and thus the $\textup{LHS of \eqref{estlemA}}$ is estimated exactly as in Subcase 1.ii.\\
\textbf{Subcase 3.iii: $|\xi_2|\sim  \xi_1$ and $\xi\gg \xi_3$.} 
Then $\xi\sim|\xi_{12}|$ and $N_{\min}=N_3$ and $N_{\max}\sim N_{\textup{med}}\sim N_1\sim N_2$. 
Therefore, the $\textup{LHS of \eqref{estlemA}}$ is estimated exactly as in Subcase 1.iii.\\
\textbf{Subcase 3.iv: $|\xi_2|\gg  \xi_1$.} Then $|\xi_2|\sim \xi_3\sim |\xi_{12}|$ and 
note that $N_{\max} = N_2$, $N_{\textup{med}}=N_3$, and $N_{\min} =N_1$. 
It follows that 
\begin{align*}
\textup{LHS of \eqref{estlemA}}& \lesssim N^{s+\delta+\frac12} N_1^{-1} N_2^{-1} 
 \big\|P_N\big(P_{N_1}V_1 P_{N_2}V_2 P_{N_3}V_3 \big)\big\|_{L_x^1}\\
 &\lesssim N_1^{-\frac12+s} N_2^{-1+\delta} 
    \|P_{N_1}V_1\|_{L_x^{\infty}} \|P_{N_2}V_2\|_{L_x^2}  \|P_{N_3}V_3\|_{L_x^2} \\
 &\lesssim N_1^sN_2^{-1-s+\delta} N_3^s   \prod_{j=1}^3 \|P_{N_j}v_j\|_{L_x^2} 
  \,.
\end{align*}
\end{proof}

For the following corollary to  the above Lemma~\ref{lem:newEstNj02} we recall that the multipliers 
$\mathfrak{m}_j^{(2)}$ depend on the dyadic numbers $N_1, N_2, N_3,N_{12}, N_{23}$, which are all dominated by $N_{\max}$.

\begin{corollary}
\label{cor:3p12}
Let $3-\sqrt{33/4}<s \le \frac14$. There exist  $\delta, \varepsilon >0$ small enough such that 
for each $j=1,2,3$, we have 
\begin{gather}
\begin{split}
& \big\| \mathcal{N}^{(2)}_{j} (\wt w, \wt u, \wt u )(t') \big\|_{L_T^1 H^{s+\delta}}\\
&\quad   \lesssim 
    N_{\max}^{-\varepsilon} 
    \Big(M^{-\frac18} + T \big(1+ \|u\|_{L_T^{\infty}H^s}\big)^{6}\Big)
  \big( \|w \|_{L_T^{\infty}H^s} + \| u \|_{L_T^{\infty}H^s}\big)  \| u\|^2_{L_{T}^{\infty}H^s}  
   \,,
\end{split}
\end{gather}
where $N_{\max}=\max\{N_1,N_2, N_3\}$. 
\end{corollary}
\begin{proof}
For the first term on the right-hand side of \eqref{2ndNFR:IbP} we simply invoke Lemma~\ref{lem:bdryterms2ndstep}.  
For the remaining terms,  
we apply Lemma~\ref{lem:newEstNj02} followed 
by Lemma~\ref{lem:lemmaA} together with Lemma~\ref{MP:lem2p7}. 
For instance, we get that for any $\theta>0$
\begin{gather}
\begin{split}
& \big\| P_N \mathcal{N}^{(2)}_{j,0} (\dt \wt w, \wt u, \wt u )(t') \big\|_{L_T^1H^{s+\delta}}\\
 &\ \  \lesssim N_{max}^{-1-s+\delta+\theta} N_{med}^s N_{min}^s  \|P_{N_1}\dt \wt w\|_{L_T^1L_x^2} 
    \|P_{N_2}\wt u\|_{L_T^{\infty}L_x^2} \|P_{N_3}\wt u\|_{L_T^{\infty}L_x^2} \\
  &\ \ \lesssim T N_1^{\frac2q+\frac12} N_{max}^{-1-s+\delta+\theta} N_{med}^s N_{min}^s  
  \big(1+\|u\|_{L_T^{\infty}H^s}\big)^6\\
 &\hspace{3cm} \times 
 \Big( \|u\|_{L_T^{\infty}H^s}   +  \|w\|_{L_T^{\infty}H^s}  \Big) 
  \|P_{N_2}u\|_{L_T^{\infty}L_x^2}  \|P_{N_3} u\|_{L_T^{\infty}L_x^2} \,,
 \end{split}
\end{gather}
where $N_{\text{med}}$ and $N_{\min}$ are as in  Lemma~\ref{lem:newEstNj02} and 
$2\le q\le 4$ satisfies the hypothesis of Lemma~\ref{lem:lemmaA}, 
i.e. 
\begin{equation}
\label{keycond1}
\left( \frac32-s\right) \left(\frac14 -\frac{1}{2q}\right) - s<0 \,.
\end{equation}
In the case that $N_{\max}=N_1$, we easily 
notice that if we also impose 
\begin{equation}
\label{keycond2}
\frac2q - \frac12 - s+\delta+\theta < -2\varepsilon \,,
\end{equation}
then 
\begin{align*}
 \big\| P_N \mathcal{N}^{(2)}_{j,0} (\dt \wt w, \wt u, \wt u )(t') \big\|_{L_T^1H^{s+\delta}} 
 &\lesssim T N_{\max}^{-\varepsilon} N^{-\varepsilon} \big(1+\|u\|_{L_T^{\infty}H^s}\big)^6 \|w\|_{L_T^{\infty}H^s}  \|u\|^2_{L_T^{\infty}H^s}
  \end{align*}
since on the support of the multiplier $\mathfrak{m}_j^{(2)}$ we have 
$ \xi<\xi_1$,  hence $N\lesssim N_1 \le N_{\max}$, for all $j=1,2,3$. 
In the case that $N_{\max}=N_2$, we clearly have 
$N_{\textup{med}}^s N_{\min}^s \le N_{\max}^s N_3^s$, and thus 
$$N_1^{\frac2q+\frac12} N_{max}^{-1-s+\delta+\theta} N_{med}^s N_{min}^s   \le 
 N_{\max}^{\frac2q - \frac12 - s+\delta+\theta} N_2^s N_3^s\lesssim
 N_{\max}^{-\varepsilon}N^{-\varepsilon}   \,,$$
by imposing the same condition \eqref{keycond2}. 
The case  $N_{\max}=N_3$ is analogous to the case  $N_{\max}=N_2$.
Lastly, the estimates for the remaining  two terms in \eqref{2ndNFR:IbP} follow similarly. 
Hence, we can pick $q$, $\delta$, and $\varepsilon$ to satisfy both \eqref{keycond1} and \eqref{keycond2} 
as long as 
\begin{equation} \label{main:cond}
s^2-6s+\frac34 <0 \,.
\end{equation}
\end{proof}

\begin{proof}[Proof of Proposition~\ref{prop:sect3}] 
It follows by gathering 
Lemmata~\ref{lem:1stresonant}, \ref{lem:bdrystep1}, \ref{lem:N2leM}, 
and Corollary~\ref{cor:3p12}. 
Indeed, we can now check that all the summation-integrals that appear 
while performing the above two steps of normal form reductions 
 are absolutely convergent. 
 Therefore, we can also regroup similar terms to derive the equation below 
 in order to provide the desired difference estimate 
\begin{gather}
\begin{split}
&\wt w_1(t) - \wt w_2(t) = \\ 
 & \int_0^t \bigg\{ \mathcal{N}^{(1)}_{\le M}(\wt w_1, \wt u_1)(t') 
   -\mathcal{N}^{(1)}_{\le M}(\wt w_2, \wt u_2)(t')\bigg\} dt' 
+ \int_0^t \big\{ \wt E_1(t') - \wt E_2(t') \big\} dt' \\
& + \Big[\mathcal{N}^{(1)}_{0}(\wt w_1, \wt u_1)(t')- \mathcal{N}^{(1)}_{0}(\wt w_2, \wt u_2)(t')\Big]_{t'=0}^{t'=t} \\
 & - \sum_{N_{12}}\sum_{N_3\lesssim N_{12}} \sum_{N_1}\sum_{N_2\lesssim N_1}
  \int_0^t\!  \bigg\{ \mathcal{N}^{(2)}_{1} (\wt w_1, \wt u_1, \wt u_1)(t') 
   -  \mathcal{N}^{(2)}_{1} (\wt w_2, \wt u_2, \wt u_2)(t') \bigg\} dt'  \\
  & - \sum_{N_{12}}\sum_{N_3\lesssim N_{12}} \int_0^t \bigg\{
   \mathcal{N}^{(1)}_{0}(\wt {E_1}_{N_{12}}, \wt {u_1}_{N_3})(t') 
  -\mathcal{N}^{(1)}_{0}(\wt {E_2}_{N_{12}}, \wt {u_2}_{N_3})(t')\bigg\} dt'\\
 &\ -\sum_{N_1}\sum_{N_{23}\lesssim N_1} \sum_{N_2}\sum_{N_3} 
   \int_0^t  \bigg\{\mathcal{N}^{(2)}_{\le M} (\wt w_1, \wt u_1, \wt u_1)(t')  
   - \mathcal{N}^{(2)}_{\le M} (\wt w_2, \wt u_2, \wt u_2)(t')  \bigg\}dt' \\
&\ + \sum_{N_1}\sum_{N_{23}\lesssim N_1}  \sum_{N_2}\sum_{N_3} \int_0^t
  \bigg\{ \mathcal{N}^{(2)}_{2} (\wt w_1, \wt u_1, \wt u_1)(t') 
   -  \mathcal{N}^{(2)}_{2} (\wt w_2, \wt u_2, \wt u_2) (t')\bigg\} dt' \\
&\ + \sum_{N_1}\sum_{N_{23}\lesssim N_1}  \sum_{N_2}\sum_{N_3} \int_0^t
  \bigg\{ \mathcal{N}^{(2)}_3 (\wt w_1, \wt u_1, \wt u_1)(t') 
   -  \mathcal{N}^{(2)}_3 (\wt w_2, \wt u_2, \wt u_2) (t')\bigg\} dt' \,.
 \end{split}
\end{gather}
By using telescoping sums and a version of Corollary~\ref{cor:3p12} 
when the difference of two solutions appears as one of the arguments, 
we handle the series containing $\mathcal{N}_{j}^{(2)}(\wt w_k, \wt u_k,\wt u_k)$, $k=1,2$, $j=1,2,3$. 
The summations containing $\wt{E_k}_{N_{12}}$, $k=1,2$, are in fact finite summations as 
$\wt{E_k}$ have compact Fourier suppport. 
Lastly, by arguing similarly to  \eqref{summationofNlesM2}, we have that 
the series containing $\mathcal{N}_{\le M}^{(2)}(\wt w_k, \wt u_k,\wt u_k)$, $k=1,2$, 
are absolutely convergent in $C_TH^{s}$. 
\end{proof}

\section{Proof of Theorem~\ref{mainthm} and Corollary~\ref{maincoro}}

\begin{proof}[Proof of Theorem~\ref{mainthm}]
Let 
$u_0\in H^s$ and let 
$u \in C(\R; H^s)$ denote the (global-in-time) solution  to BO with initial data $u_0$ 
provided by the results in \cite{MolinetPilodAPDE12} or \cite{IonescuKenigBO} or \cite{IfrimTataruBO}. 
Suppose there exists another solution 
$u^{\dagger} \in C(I; H^s)$ to BO (not necessarily global-in-time), with the same initial data $u_0$, 
on some open time interval $I$, neighborhood of $t=0$. 
By the time translation symmetry of BO, we can assume without loss of generality that 
$\max \{ t\in I : u(t) = u^{\dagger}(t)\} = 0 $ 
and thus to reach a contradiction, it suffices to show that $u = u^{\dagger}$ in $C_TH^s$ for any small $T>0$.   
By the time reversal symmetry of BO, one argues analogously for negative times.  

Denote by $F, w$ and $F^{\dagger}, w^{\dagger}$ 
the corresponding spatial antiderivatives and gauge transformations of $u$ and $u^{\dagger}$, respectively. 
We fix some $T'\in I\cap (0,1)$ and 
we set 
$$\wt{K}:=(1+C_2) \big(1+\|u\|_{C_{T'}H^s} + \|u^{\dagger}\|_{C_{T'}H^s}\big)^2<\infty\,,$$ 
where $C_2$ denotes the implicit constant in \eqref{diffinus:est2}. 
By first choosing $N\in 2^{\Z}$ such that 
$$C_2 \wt{K}^2\big(N^{s-\frac12} + \|P_{>\frac{N}{2}} w^{\dagger}\|_{C_{T'}H^s} \big)\le \frac14$$ 
and then by choosing $0<T<T'$ such that 
$$C_1 \wt{K} T N^{\frac32 +s} \le \frac14\,,$$
where $C_1$ is the implicit constants in
\eqref{diffinus:est1}, 
Lemma~\ref{lem:diffinus} implies that 
\begin{align}
\label{pfmainthm:star}
\|u - u^{\dagger}\|_{C_TH^s} \le 2C_2 \wt{K} \|w - w^{\dagger}\|_{C_{T}H^s}\,.
\end{align}
Since both $w$ and $w^{\dagger}$ satisfy the integral formulation of \eqref{2ndNFR:IbP}, 
we can appeal to Proposition~\ref{prop:sect3} and thus 
there is some $C_3>0$ such that 
\begin{align}
\label{pfmainthm:interm}
\| w- w^{\dagger}\|_{C_TH^s} \le 
  C_3   \big( T M^{\frac32} + M^{-\frac{1}{16}} \big) \wt{K}^{10}
    \big( \| w- w^{\dagger}\|_{C_TH^s} + \| u- u^{\dagger}\|_{C_TH^s} \big) \, .
\end{align}
With $\beta\in(0,1)$ such that 
\begin{align}
\label{pfmainthm:beta} 
 2 C_2 \wt{K} \frac{\beta}{1-\beta} \le \frac12 \,,
\end{align}
choose $M\gg 1$ such that 
\begin{align*}
C_3 M^{-\frac{1}{16}}  \wt{K}^{10} \le \frac{\beta}{2} 
\end{align*}
and then we adjust $T$  such that we also verify 
\begin{align*}
C_3  T M^{\frac32}  \wt{K}^{10} \le \frac{\beta}{2}  \,.
\end{align*}
Then, from \eqref{pfmainthm:interm} we have
\begin{align}
\label{pfmainthm:dstar}
\| w- w^{\dagger}\|_{C_TH^s} \le 
  \frac{\beta}{1-\beta} \| u- u^{\dagger}\|_{C_TH^s}
\end{align}
Hence, by \eqref{pfmainthm:star}, \eqref{pfmainthm:beta}, and  \eqref{pfmainthm:dstar}, 
we get 
$ \| u- u^{\dagger}\|_{C_TH^s}  = 0  $, which completes the proof of Theorem~\ref{mainthm}.
\end{proof}

\begin{proof}[Proof of Corollary~\ref{maincoro}] 
We recall here that $w= e^{it\dx^2} \wt w$ and that $\wt w$ satisfies  the normal form equation  
\begin{gather}
\begin{split}
\wt w(t) - \wt w_0 = 
 & \int_0^t  \mathcal{N}^{(1)}_{\le M}(\wt w, \wt u)(t') dt' 
+ \int_0^t  \wt E(t') dt' + \Big[\mathcal{N}^{(1)}_{0}(\wt w, \wt u)(t')\Big]_{t'=0}^{t'=t} \\
 & - \sum_{N_{12}}\sum_{N_3\lesssim N_{12}} \sum_{N_1}\sum_{N_2\lesssim N_1}
  \int_0^t  \mathcal{N}^{(2)}_{1} (\wt w, \wt u, \wt u)(t')  dt'  \\
  & - \sum_{N_{12}}\sum_{N_3\lesssim N_{12}} \int_0^t 
   \mathcal{N}^{(1)}_{0}(\wt {E}_{N_{12}}, \wt {u}_{N_3})(t') dt'\\
 &\ -\sum_{N_1}\sum_{N_{23}\lesssim N_1} \sum_{N_2}\sum_{N_3} 
   \int_0^t \mathcal{N}^{(2)}_{\le M} (\wt w, \wt u, \wt u)(t')  dt' \\
&\ + \sum_{N_1}\sum_{N_{23}\lesssim N_1}  \sum_{N_2}\sum_{N_3} \int_0^t
   \mathcal{N}^{(2)}_{2} (\wt w, \wt u, \wt u)(t') dt' \\
&\ + \sum_{N_1}\sum_{N_{23}\lesssim N_1}  \sum_{N_2}\sum_{N_3} \int_0^t
  \mathcal{N}^{(2)}_3 (\wt w, \wt u, \wt u)(t') dt' \,.
 \end{split}
\end{gather}
Therefore, by setting
$M:=T^{-\frac{16}{25}}$,
by Lemmata ~\ref{lem:1stresonant}, \ref{lem:bdrystep1}, \ref{lem:N2leM}, 
and Corollary~\ref{cor:3p12}, we get 
\begin{align*}
& \|w(t) - e^{it\dx^2} w_0\|_{H^{s+\delta}} = \|\wt w(t) - \wt w(0)\|_{H^{s+\delta}}
 \lesssim  T^{\frac{1}{25}} \big(1+\|u\|_{C_{T}H^s} \big)^{20} \| u\|_{C_TH^s} \,.
\end{align*}
Due to the uniqueness result of Theorem~\ref{mainthm} and since $\|u\|_{C_TH^s}\le C(T, \|u_0\|_{H^s})$ (see for example Theorem~1.1 in \cite{IonescuKenigBO}), we conclude the proof of Corollary~\ref{maincoro}. 
\end{proof}

\section{The periodic case} \label{Sec:periodic}

Here we consider the Benjamin-Ono equation \eqref{BO} posed on the torus $\T:=\R/2\pi \Z$. We point out the main modifications needed to obtain the unconditional uniqueness result of Theorem~\ref{mainthm}. 

\subsection{The gauge transformation}
Since the Benjamin-Ono evolution conserves the mean, i.e. 
$\int_{\T}u(t,x)dx = \int_{\T}u_0(x)dx$ for all $t$, by using the translation transformation 
$$\wt u(t,x) := u\Big(t,x- \frac{t}{\pi}\int_{\T} u_0\Big)-\frac1{2\pi}\int_{\T} u_0 , $$
we can assume without loss of generality  
that 
$$\int_{\T} u(t,x) dx =0\ ,\ \text{ for all } t\,.$$ 
We then define $F:=\dx^{-1}u$ the spatial anti-derivative of $u$ by 
$$  \widehat{F}(0) =0 \quad ,\quad \widehat{F}(n)= \frac{1}{i n} \widehat u(n) \ ,\ n\in\Z^* $$
and note that in place of Lemma~\ref{lem:diffFs}, we easily have 
$\|F_1-F_2\|_{L^{\infty}(\T)} \lesssim \|u_1-u_2\|_{L^2(\T)}$ 
with a constant independent of $t$. 
Moreover, we denote 
\[ P_{\pm}(f)= \mathcal{F}^{-1}\left( \textbf{1}_{\pm n>0} \widehat{f}\right) \quad \text{and} \quad P_0(f)= \mathcal{F}^{-1}\left( \textbf{1}_{n=0} \widehat{f}\right)=\widehat{f}(0) . \]
Since $\widehat{e^{-iF}}(n)$ is well-defined for all $n\in \Z$, 
 $P_{\pm}(e^{-iF})$ are well-defined $L^2(\T)$-functions with
$\|e^{-iF}\|_{L^2(\T)} \sim 1$. 
Also, note that $F$ satisfies
\begin{equation*}
\dt F+\H\dx^2 F = u^2 - P_0(u^2)\, .
\end{equation*}

The gauge transformation
\begin{equation}
\label{gaugetrT}
w:= \dx P_{+} (e^{-iF}) \,,
\end{equation}
satisfies
\begin{equation}
\label{eq:dtw_per.1}
\dt w - i\dx^2 w   -i m(t) w = -2P_+\dx\big[ \dx^{-1}w \cdot P_-\dx u\big]  \, ,
\end{equation}
where $m(t):= P_0(u^2)(t)$.

It is worth noting that at this level of regularity, the weak solutions of BO, which are not constructed through (conditional) well-posedness results, do not necessarily satisfy the $L^2$-conservation law. For this reason, we perform another gauge transformation $:w \mapsto e^{-i\int_0^tm(t')dt'}w$, so that the equation for the gauge variable becomes
\begin{equation}
\label{eq:dtw_per}
\dt w - i\dx^2 w = -2P_+\dx\big[ \dx^{-1}w \cdot P_-\dx u\big]  \, ,
\end{equation}
and, by setting $E[f,g]:= -2P_+\Plo\dx[f \cdot P_-\dx g]$,  we rewrite \eqref{eq:dtw_per} precisely as \eqref{eqw-1}. 
It is easy to check that the estimate corresponding to \eqref{est:lem2p7} also holds in this case.

\subsection{The Strichartz estimates}

\begin{lemma}[refined Strichartz estimates on the torus]
Let $0\le s\le \frac14$, $N\in 2^{\Z_+}$, $T>0$,  and $2\le p\le 4$. Let $u$ be a solution to \eqref{BO:lin} with $F=\partial_x(u_1u_2)$. Then, we have 
\begin{equation}
\label{refinedStrichartz_per}
 \|P_N u\|_{L^p([0,T]\times\T)} \lesssim 
 T^{\frac1p} N^{\beta(s,p)} \Big( \|P_N u\|_{L_T^{\infty}H_x^s}  
 + \|u_1\|_{L_T^{\infty}H_x^s}  \|u_2\|_{L_T^{\infty}H_x^s} \Big) \,,
\end{equation}
where
\begin{equation}
\label{defn:betasp}
\beta(s,p): =\big(\frac32-s\big) \big(\frac14 - \frac{1}{2p}\big) - s\,.
\end{equation}
\end{lemma}
\begin{proof}
Following the proof of \cite[Lemma~2.1]{MolinetRibaud}, we use the $L^4$-Strichartz estimate due to Zygmund \cite{Zygmund74} to deduce 
\begin{equation*}
\big\| e^{t\H\dx^2} f \big\|_{L^4([0,T]\times\T)} \lesssim T^{\frac18} \|f\|_{L_x^2(\T)} \,.
\end{equation*}
After interpolating with the trivial estimate 
\begin{equation*}
\big\| e^{t\H\dx^2} f \big\|_{L^2([0,T]\times\T)} \lesssim T^{\frac12} \|f\|_{L_x^2(\T)} \,, 
\end{equation*}
this implies 
\begin{equation} \label{strich:BO_per}
\big\| e^{t\H\dx^2} f \big\|_{L^p([0,T]\times\T)} \lesssim T^{\frac3{2p}-\frac14} \|f\|_{L_x^2(\T)} \,,
\end{equation}
for any $2 \le p \le 4$.

The proof then follows similarly to the proof of Lemma~\ref{lem:refStrichartz}. With $\delta>0$ to be chosen later, 
let $I_j = :[a_j, b_j]$ be such that $\bigcup_j I_j = [0,T]$, $b_j-a_j\sim N^{-\delta}$, and 
the number of such intervals is $\sim TN^{\delta}$. 
We then deduce from \eqref{strich:BO_per} 
\begin{align*}
\|P_N u\|^p_{L^p([0,T]\times\T)} &= \sum_j \int_{a_j}^{b_j} \|P_N u\|_{L_x^p}^p dt  \\
&\lesssim  TN^{\delta(1-\frac32+\frac{p}4)} \|P_N u\|_{L_T^{\infty}L_x^2}^p + \sum_j |I_j|^{p-1}|I_j|^{\frac32-\frac{p}4}  \|P_N F\|_{L_{I_j}^pL_x^2}^p,
\end{align*}
so that
\begin{equation*}
\|P_N u\|_{L^p([0,T]\times \mathbb T)} \lesssim 
 T^{\frac1p} N^{(-\frac{1}{2p}+\frac14)\delta} \|P_N u\|_{L_T^{\infty}L_x^2}  
 +T^{\frac1p} N^{-\big( \frac34+ \frac{1}{2p}\big)\delta} \|P_NF\|_{L_T^{\infty}L_x^2} \,.
\end{equation*}


\noindent
In particular, for 
\begin{equation*}
F=\dx(u_1 u_2)\,,
\end{equation*}
 we get
\begin{equation*}
\|P_N u\|_{L^p([0,T]\times \mathbb T)} \lesssim 
 T^{\frac1p} N^{(-\frac{1}{2p}+\frac14)\delta-s} \|P_N u\|_{L_T^{\infty}H_x^s}  
 +T^{\frac1p} N^{1-\big( \frac34+ \frac{1}{2p}\big)\delta} \|P_N(u_1u_2)\|_{L_T^{\infty}L_x^2} \,.
\end{equation*}
Together with 
\begin{equation*}
 \|P_N(u_1u_2)\|_{L_x^2} \lesssim  N^{\frac1r - \frac12} \|u_1u_2\|_{L_x^r}
   \le N^{\frac1r - \frac12}  \|u_1\|_{L_x^{2r}} \|u_2\|_{L_x^{2r}} 
 \lesssim N^{\frac1r - \frac12} \|u_1\|_{H_x^s} \|u_2\|_{H_x^s} \,,
\end{equation*}
where $1\le r\le 2$ is determined by $s=\frac12 - \frac{1}{2r}$, or equivalently $r= \frac{1}{1-2s}$, 
we obtain
\begin{equation*}
\|P_N u\|_{L^p([0,T]\times \mathbb T)} \lesssim 
 T^{\frac1p} N^{(-\frac{1}{2p}+\frac14)\delta-s} \|P_N u\|_{L_T^{\infty}H_x^s}  
 +T^{\frac1p} N^{\frac32 -\big( \frac34+ \frac{1}{2p}\big)\delta-2s} 
 \|u_1\|_{L_T^{\infty}H_x^s}  \|u_2\|_{L_T^{\infty}H_x^s} 
\end{equation*}
(the restriction on $r$ imposes $0\le s\le \frac14$). 
We choose $\delta$ such that 
$$(-\frac{1}{2p}+\frac14)\delta-s = \frac32 -\big( \frac34+ \frac{1}{2p}\big)\delta-2s, $$ or equivalently $\delta = \frac32-s$, 
and with 
$\beta(s,p)$ as in \eqref{defn:betasp},
we obtain 
\eqref{refinedStrichartz_per}. 
\end{proof}

\subsection{Nonlinear estimates and proofs of the main results}

We note that the normal form transformations as well as the nonlinear estimates 
from Section~\ref{normalform:sec} carry over exactly as in the real-line case. 
The proof of the estimates of Lemma~\ref{lem:lemmaA} is similar, 
but now using the refined Strichartz estimate \eqref{refinedStrichartz_per} instead of \eqref{refinedStrichartz}. 
For the  convenience of the reader we check the numerology of Corollary~\ref{cor:3p12} and verify that the same regularity condition is necessary. 
Indeed, we must ensure 
\begin{align*}
\beta(s,p)<0 \text{ and } \delta+\frac2p -\frac12 < s
\end{align*}
(to be compared with \eqref{keycond1} and \eqref{keycond2} in the real-line case), or equivalently
$$ (\frac32-s)(\frac14-\frac{1}{2p}) -s<0 \text{ and } \frac14-\frac{1}{2p}> \frac18-\frac{s}{4}$$
which hold true for some $p\in (2,4)$ under the same condition \eqref{main:cond}. 

The proof of Theorem~\ref{mainthm} is then carried out by reasoning on equation~\eqref{eq:dtw_per} and by arguing exactly as on the real line case.

Finally, we comment on the proof of Corollary~\ref{maincoro} in the periodic case. With Theorem~\ref{mainthm} in hand, one can assume that any solution at this level of regularity is constructed through a (conditional) well-posedness result. In particular, by invoking the $L^2$-conservation law, we can then rigorously justify that  $m(t)=m_0:=P_0(u_0^2)$, so that the equation for the gauge variable \eqref{eq:dtw_per.1} becomes
\begin{equation} 
\label{eq:dtw_per.2}
\dt w - i\dx^2 w  -im_0 w = -2P_+\dx\big[ \dx^{-1}w \cdot P_-\dx u\big]  \, .
\end{equation}
 Since the $L^p([0,T]\times\T)$-norms are invariant under 
the transformation $:w\mapsto e^{im_0t} w$, 
one also has the same Strichartz estimate for $P_N w$ as in Lemma~\ref{defn:alphasp} (ii). Therefore, we conclude the proof of Corollary~\ref{maincoro} in the periodic case by arguing exactly as in the continuous case.

\medskip
\section*{Acknowledgements} 

The authors were supported by a Trond Mohn Foundation grant. R.M. would like to thank Tadahiro Oh for suggesting the problem.  The authors are very grateful to the anonymous referees for many important comments and suggestions which greatly improved the present version. In particular, following one of the referees' suggestions, the proof of Proposition \ref{prop:sect3} has been reorganised much more concisely.

\medskip

\bibliographystyle{siam}
\bibliography{}

\end{document}